\documentclass[11pt]{article}
\usepackage[ansinew]{inputenc}
\usepackage{graphicx}
\usepackage{color}

\usepackage{enumerate,latexsym}
\usepackage{latexsym}
\usepackage{amsmath,amssymb}
\usepackage{graphicx}
\newfont{\bb}{msbm10 at 11pt}
\newfont{\bbsmall}{msbm8 at 8pt}
\def\rth{\mathbb{R}^3}
\def\R{\mathbb{R}}
\def\B{\mathbb{B}}
\def\N{\mathbb{N}}

\def\D{\mathbb{D}}

\def\esf{\mathbb{S}}

\newcommand{\ben}{\begin{enumerate}}
\newcommand{\bit}{\begin{itemize}}
\newcommand{\een}{\end{enumerate}}
\newcommand{\eit}{\end{itemize}}
\newcommand{\wh}{\widehat}
\newcommand{\Int}{\mbox{Int}}

\newcommand{\wt}{\widetilde}

\def\a{{\alpha}}
\def\lc{{\cal L}}

\def\g{{\gamma}}
\def\G{{\Gamma}}
\def\l{{\lambda}}

\def\de{{\delta}}
\def\be{{\beta}}
\def\ve{{\varepsilon}}
\def\cF{{\cal F}}
\def\cS{{\cal S}}

\def\centerbmp#1#2#3{\vskip#2\relax\centerline{\hbox to#1{\special
    {bmp:#3 x=#1, y=#2}\hfil}}}

\newtheorem{theorem}{Theorem}[section]
\newtheorem{lemma}[theorem]{Lemma}
\newtheorem{proposition}[theorem]{Proposition}

\newtheorem{remark}[theorem]{Remark}
\newtheorem{corollary}[theorem]{Corollary}
\newtheorem{definition}[theorem]{Definition}

\newtheorem{assertion}[theorem]{Assertion}

\newenvironment{proof}{\smallskip\noindent{\it Proof.}\hskip \labelsep}
{\hfill\penalty10000\raisebox{-.09em}{$\Box$}\par\medskip}

\begin{document}
\begin{title}
{The classification of CMC foliations of $\R^3$ and $\esf ^3$ with countably many singularities}
\end{title}
\vskip .2in

\begin{author}
{William H. Meeks III\thanks{This material is based upon
   work for the NSF under Award no. DMS - 1309236.
   Any opinions, findings, and conclusions or recommendations
   expressed in this publication are those of the authors and do not
   necessarily reflect the views of the NSF.}
   \and Joaqu\'\i n P\' erez
\and Antonio Ros\thanks{The second and third authors were supported in part
by the MEC/FEDER grant no. MTM2011-22547, and the
Regional J. Andaluc\'\i a grant no. P06-FQM-01642.}}
\end{author}
\maketitle
\begin{abstract}
In this paper we generalize the Local Removable Singularity
Theorem in~\cite{mpr10} for minimal laminations to the case of weak
$H$-laminations (with $H\in \R $ constant) in a punctured ball of a
Riemannian three-manifold.
We also obtain a curvature estimate for any weak CMC foliation
(with possibly varying constant mean curvature from leaf to leaf)
of a compact Riemannian three-manifold $N$ with boundary solely
in terms of a bound of the absolute sectional curvature of $N$ and of
the distance  to the boundary of $N$.
We then apply these results to classify weak CMC foliations
of $\R^3$ and $\esf^3$ with a closed countable set of singularities.

\vspace{.3cm}


\end{abstract}

\section{Introduction.}

In this paper, we address a number of outstanding
classical questions on the geometry of embedded
surfaces of constant mean curvature and more generally, laminations
and foliations of $\R^3$ and other three-manifolds, where the leaves
of these laminations are surfaces with constant mean curvature
(possibly varying from
leaf to leaf). In the foliation case, we call every such foliation a
{\it CMC foliation.} The first of these classical problems is to
classify codimension one CMC foliations of
$\R^3$ or $\esf^3$ (with their standard metrics) in the complement of a closed countable
set ${\cal S}$. The simplest such examples in  $\R^3$ are families of
parallel planes or concentric spheres around a given point. A
slightly more complicated example appears when considering a
family of pairwise disjoint planes and spheres as in
Figure~\ref{figspheres}, where the set ${\cal S}$ consists of two points.
\begin{figure}
\begin{center}
\includegraphics[width=12cm]{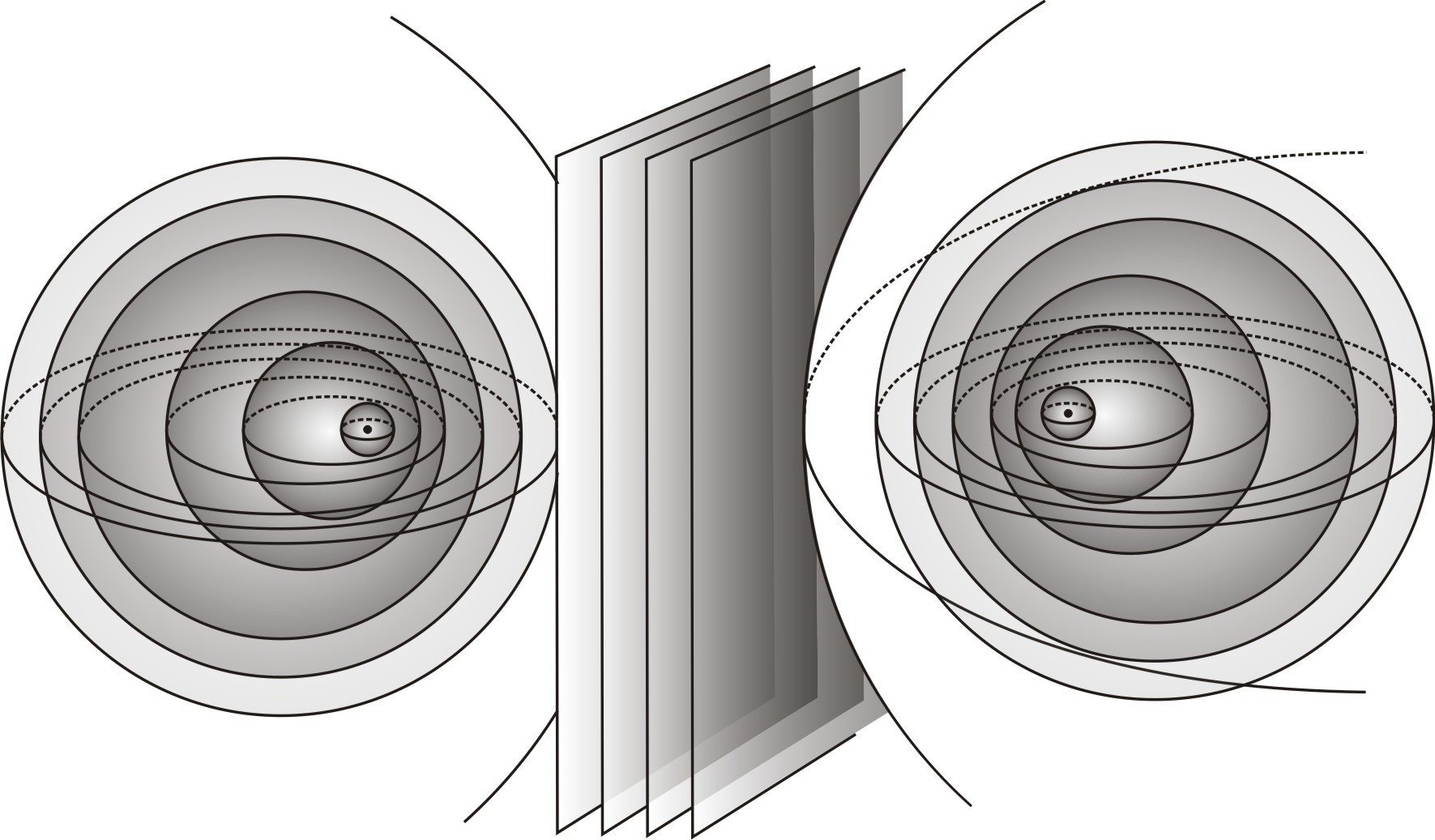}
\caption{A foliation of $\R^3$  by spheres and planes with two
singularities. }
\label{figspheres}
\end{center}
\end{figure}
We solve this classification problem in complete generality (see
Theorem~\ref{thmspheres} for a solution of an even more general
problem, where the leaves of the "foliation" are allowed to
intersect in a controlled manner\footnote{The quotes here refer to
the notion of {\it weak CMC foliation;} see Definition~\ref{definition}.}):

\begin{theorem}
\label{thmspheresintrod}
Suppose ${\cal F}$ is a CMC foliation of $\rth$ with a  closed
countable set $ \cal S$ of singularities\footnote{We mean that ${\cal F}$ is a foliation
of $\rth-{\cal S}$, and it does not extend to a CMC foliation of
$\rth-{\cal S}'$ for any proper closed subset ${\cal S}'\subset
{\cal S}$.}
Then, all leaves of ${\cal F}$ are contained in
planes and round spheres. Furthermore if $\cal S$ is empty, then $\cal F$ is a foliation by planes.
\end{theorem}
In the case of the unit three-sphere $\esf^3\subset \R^4$ with its standard metric of constant sectional 
curvature, we obtain a similar result:
\par
\vspace{.2cm}
{\it The leaves of every CMC foliation   of $\esf^3$ with a closed
countable set $ \cal S$ of singularities are contained in round spheres,
and ${\cal S}$ is always non-empty.
}
\par
\vspace{.2cm}

We remark that the special case of Theorem~\ref{thmspheresintrod}
where the singular set ${\cal S}$ of ${\cal F}$ is empty is a
classical result of Meeks~\cite{me17}. Also, we note that in the
statement of the above theorem, we made no assumption on the
regularity of the foliation ${\cal F}$. However, the proofs in this
paper require that ${\cal F}$ has bounded second fundamental
form on compact sets of $N=\R^3$ or $\esf^3$ minus the singular set
${\cal S}$. This bounded curvature assumption always holds
for a topological CMC foliation\footnote{See Definition~\ref{deflamination}
for this concept.} by
%
%
recent work of Meeks and Tinaglia
\cite{mt7} on curvature estimates for embedded, non-zero constant
mean curvature disks and a related 1-sided curvature estimate for embedded surfaces of any constant mean curvature;  in
the case that all of the leaves of the lamination of a three-manifold are
minimal, this 1-sided curvature estimate was given earlier by
Colding and Minicozzi~\cite{cm23}.

Consider a foliation ${\cal F}$ of a Riemannian three-manifold $N$
with leaves having constant absolute mean curvature.
After possibly passing to a four-sheeted cover, we can assume $N$ is oriented and that all
leaves of ${\cal F}$ are oriented consistently, in the sense that
there exists a continuous, nowhere zero vector field in $N$ which is
transversal to the leaves of ${\cal F}$. In this situation, the mean
curvature function of the leaves of ${\cal F}$ is well-defined and so $\cF$ is a CMC foliation.
Therefore, when analyzing the structure of such a CMC foliation ${\cal F}$,
it is natural to consider for each $H\in \R $,
the subset ${\cal F}(H)$ of ${\cal F}$ of those leaves that have
mean curvature $H$. Such a subset ${\cal F}(H)$ 
is closed since the mean curvature function is continuous on
${\cal F}$; ${\cal F}(H)$ is an example of an $H$-lamination. A
cornerstone in proving the above classification results is to analyze the
structure of an $H$-lamination ${\cal L}$ (or more generally, a weak
$H$-lamination, see Definition~\ref{definition}) of a punctured ball
in a Riemannian three-manifold, in a small neighborhood of the
puncture. This local problem can be viewed as a desingularization
problem. In our previous paper~\cite{mpr10}, we characterized the removability of
an isolated singularity of a minimal lamination ${\cal L}$ of a punctured
ball in terms of the growth of the norm of the second fundamental form of
the leaves of ${\cal L}$ when extrinsically approaching the puncture. We will
extend this Local Removable Singularity Theorem to the case of a
weak $H$-lamination, see Theorem~\ref{tt2} below.

Next we set some specific notation to be used throughout the paper, which is necessary
to state the next Local Removable Singularity Theorem.
Given a Riemannian three-manifold $N$ and a point $p\in N$, we  denote
by $d_N$ the distance function in $N$ and by $B_N(p,r), \overline{B}_N(p,r)$, $S_N^2(p,r)$
the open metric ball of center $p$ and radius $r>0$, its closure and boundary sphere,
respectively. In the case $N=\R^3$, we use the notation $\B (p,r)=B _{\R^3}(p,r)$, $\esf^2(p,r)=
S^2_{\R^3}(p,r)$ and $\B (r)=\B (\vec{0},r)$, $\esf^2(r)=\esf^2(\vec{0},r)$,
where $\vec{0}=(0,0,0)$. Furthermore,
$R\colon\R^3\to \R $ stands for the distance function to
the origin $\vec{0}$.
For a codimension-one lamination  $\lc$ of $N$ and a leaf $L$ of
${\cal L}$, we denote by $|\sigma _{L}|$
the norm of the second fundamental form of $L$. Since
leaves of ${\cal L}$ do not intersect, it makes sense
to consider the norm of the second fundamental form as a
function defined on the union of the leaves of
${\cal L}$, which we denote by
 $|\sigma _{{\cal L}}|$.
In the case of a weak $H$-lamination $\lc $ of $N$, given $p\in {\cal L}$ there
exist at most two leaves of ${\cal L}$ passing through $p$ (by the maximum
principle for constant mean curvature surfaces), and thus, we define
$|\sigma _{\cal L}|$ to be the function on ${\cal L}$ that assigns to each
$p\in {\cal L}$ the maximum of the norms of the second fundamental forms of leaves
$L$ of ${\cal L}$ such that $p\in L$. Observe that $|\sigma _{\cal L}|$ is
not necessarily continuous.

\begin{theorem} [Local Removable Singularity Theorem]
\label{tt2}
A weak $H$-lamination ${\cal L}$ of a punctured ball $B_N(p,r)-\{ p\} $ of a
Riemannian three-manifold $N$ extends to a weak $H$-lamination
of $B_N(p,r)$ if and only if there exists a positive constant $C$
such that $|\sigma _{{\cal L}}|\, d_N(p,\cdot )\leq C$ in some subball.
In particular under this hypothesis,
\begin{enumerate}
\item The second fundamental form of  ${\cal L}$ is bounded in a neighborhood of $p$.
\item If ${\cal L}$ consists of a single leaf $M\subset B_N(p,r)-\{ p\} $
which is a properly immersed weak $H$-surface\footnote{We mean here that
$M$ is allowed to intersect itself only in the way that the leaves of a weak
$H$-lamination might intersect, see Definition~\ref{definition}.}
with $\mbox{\rm \O }\neq \partial M\subset \partial B_N(p,r)$, then
$M$ extends smoothly across $p$.
\end{enumerate}
\end{theorem}
We remark that in the case $H=0$, a weak
$H$-lamination is just a minimal lamination, see the first paragraph just after
Definition~\ref{definition}. In this way, Theorem~\ref{tt2} generalizes
the minimal case of the Local Removable Singularity Theorem proven in~\cite{mpr10}.

Besides the above Local Removable Singularity Theorem, a second key ingredient is needed in
the proof of Theorem~\ref{thmspheresintrod}: a universal scale-invariant curvature estimate valid
for any CMC foliation of a compact Riemannian three-manifold with boundary, solely in terms
of an upper bound for its sectional curvature. The next result is inspired by
previous curvature estimates by Schoen~\cite{sc3} and Ros~\cite{ros9} for
compact stable minimal surfaces with boundary, and by
Rosenberg, Souam and Toubiana~\cite{rst1} for stable constant mean curvature surfaces.

\begin{theorem}[Curvature Estimates for CMC foliations]
\label{thm5.7}
There exists a positive constant $A>0$ such that the following statement holds.
Given $\Lambda \geq 0$, a compact Riemannian three-manifold $N$ with boundary
whose absolute sectional curvature is at most $\Lambda $, a weak $CMC$ foliation
${\cal F}$ of $N$ and a point $p\in \rm Int(N)$, we have
\[
|\sigma_{\cal F}|(p)\leq \frac{A}{\min\{{\rm dist}_N(p,\partial N),
\frac{\pi }{\sqrt{\Lambda }}\}},
\]
where $|\sigma _{\cal F}|\colon N\to [0,\infty )$ is the function that assigns to each $p\in N$
the supremum of the norms of the second fundamental forms of leaves of ${\cal F}$ passing through
$p$.
\end{theorem}

The above curvature estimate is an essential tool for analyzing the structure of a weak CMC
foliation of a small geodesic Riemannian three-ball punctured at its center.  Among other things we
prove that if the mean curvatures  of the leaves of such a weak CMC foliation  are bounded in a neighborhood
of the puncture, then  the weak CMC foliation extends
across the puncture to a weak CMC foliation of the ball.

A global application of the above theorem is that a compact, orientable Riemannian three-manifold not diffeomorphic to
the three-sphere $\esf^3$ or to the real projective three-space $\mathbb{P}^3$
does not admit any weak CMC foliation with a non-empty countable closed set of singularities; see~\cite{mpe14} for this and other related results.
In a different  direction, it is natural to ask whether every closed, orientable three-manifold admits a Riemannian metric together with a
smooth CMC foliation. The existence of such foliations follows from the next main theorem in~\cite{mpe13} and the facts that
every closed three-manifold has a differentiable structure and the Euler characteristic of
any closed manifold of odd dimension is always zero; the proof of this theorem
relies heavily on the seminal works of Thurston~\cite{th3}
on the existence of smooth codimension-one foliations of smooth closed $n$-manifolds and  of
 Sullivan~\cite{sul1}, which explains  when such foliations are minimal with respect to some Riemannian metric.

\begin{theorem} \label{main}
A smooth closed orientable $n$-manifold $X$ admits a smooth CMC foliation $\cF$
for some Riemannian metric  if and only if its Euler characteristic  is zero. Furthermore,
$\cF$ can be chosen so that the mean curvature function of its leaves changes sign.
\end{theorem}

The paper is organized as follows. In Section~\ref{sec3} we
extend to $H$-surfaces the Stability Lemma proved in~\cite{mpr10} for the minimal case.
In Section~\ref{sec10} we discuss some regularity aspects of $H$ and CMC laminations,
and define weak $H$ and CMC laminations. Section~\ref{sec4} contains the proof of Theorem~\ref{thm5.7}.
Section~\ref{secprooftt2} is devoted
to prove the Local Removable Singularity Theorem~\ref{tt2} based
on the previously proven minimal case (Theorem~1.1 in~\cite{mpr10}). In Section~\ref{sectspheres}
 we demonstrate Theorem~\ref{thmspheresintrod}. Finally,
in Section~\ref{sec:localstructure} we apply a rescaling argument and
Theorem~\ref{thmspheresintrod} to obtain Theorem~\ref{corol6.6} which describes
the structure of any singular CMC foliation of a Riemannian three-manifold in
a small neighborhood of any of its isolated singular points.



\section{Stable surfaces with constant mean curvature which are complete outside of a point.}
\label{sec3}

In~\cite{mpr10} we extended the classical characterization of planes
as the only orientable, complete, stable
minimal surfaces in $\R^3$ (do Carmo and Peng~\cite{cp1},
Fischer-Colbrie and Schoen~\cite{fs1},
Pogorelov~\cite{po1}) to the case that completeness is only required outside
a point, and called this result the {\it Stability Lemma}\footnote{
This minimal stability lemma was found independently by Colding and Minicozzi~\cite{cm25}.}.
This result was a crucial tool in the proof of the minimal case of the Local Removable Singularity
Theorem (Theorem~~1.1 in~\cite{mpr10}). Next we extend the Stability Lemma
to the case of constant mean curvature surfaces, although this extension will not
be directly used in the proof of Theorem~\ref{tt2}: rather than this, the proof of
Theorem~\ref{tt2} will be based on the validity of the minimal case of the
Local Removable Singularity Theorem proven in~\cite{mpr10}, which in turn only
uses the minimal case of the Stability Lemma. We remark that the
notion of stability used in this paper is that the
first eigenvalue of the stability operator for
compactly supported variations is non-negative, in contrast with
the usual stability notion related to the isoperimetric problem,
where only compactly supported variations that preserve
infinitesimally the volume are considered.

\begin{definition}
\label{defcompl}
{\rm An immersed surface $M\subset \R^3-\{ \vec{0}\} $ is
{\it complete outside the origin,} if every divergent path in $M$ of
finite length has as limit point the origin. }
\end{definition}

Let $x\colon M\rightarrow \R^3$ be an isometric immersion of an orientable surface
$M$. The mean curvature function $H$ of $x$ is constant with value $c\in \R $ (in short, $x(M)$ is an immersed
$c$-surface) if and only if for all smooth compact subdomains of $M$, $x$ restricted to the subdomain
is a critical point of the functional $\mbox{Area}-2c\, \mbox{Volume}$.
Given an $H$-surface
$x(M)$ and a function $f\in C^{\infty }_0(M)$, the second variation of
$\mbox{Area}-2H\, \mbox{Volume}$ for any compactly supported variation of $x$ whose normal
component of the variational field is $f$, is well-known to be
\[
 Q(f,f)=-\int _M fLf\, dA ,
 \]
where $L= \Delta +|\sigma |^2=\Delta -2K+4H^2$ is the {\it Jacobi operator} on $M$  (here $|\sigma |$
is the norm of the second fundamental form and $K$ is the Gaussian curvature).
The immersion is said to be {\it stable}
if $-L$ is a non-negative operator on $M$, i.e., $Q(f,f)\geq 0$ for every $f\in C^{\infty }_0(M)$.
More generally, a Schr\"{o}dinger operator $-(\Delta +q)$ with $q\in C^{\infty }(M)$ is
called {\it non-negative} if
\[
\int _M (|\nabla f|^2-qf^2)\geq 0, \quad \mbox{for all }f\in C^{\infty }_0(M).
\]

\begin{lemma} [Stability Lemma for $H$-surfaces]
\label{lema1}
Let $M\subset \R^3-\{ \vec{0}\} $ be a stable, immersed
constant mean curvature (orientable if minimal) surface,
which is complete outside the origin. Then, the closure $\overline{M}$ of $M$ is a plane.
\end{lemma}
\begin{proof}
If $\vec{0}\notin \overline{M}$, then $M$ is complete and so, it is
a plane~\cite{cp1,fs1,po1}.
Assume now that $\vec{0}\in\overline{M}$. Consider the
metric $\widetilde{g}=\frac{1}{R^2}g$ on $M$, where $g$ is the
metric induced by the usual inner product $\langle ,\rangle $ of
$\R^3$ and $R$ is the distance to the origin in $\R^3$.
Note that if $M$ were a plane through $\vec{0}$, then
$\widetilde{g}$ would be the metric on $M-\{ \vec{0}\} $ of an  infinite cylinder
of radius 1 with ends at $\vec{0}$ and at infinity.
Since $(\R^3-\{ \vec{0}\} ,\widehat{g})$ with
$\widehat{g}=\frac{1}{R^2}\langle ,\rangle $, is isometric to $\esf
^2(1)\times \R $, then $(M,\widetilde{g})\subset (\R^3-\{ \vec{0}\}
,\widehat{g})$ is complete.

The laplacians and Gauss curvatures of $g,\widetilde{g}$ are related
by the equations $\widetilde{\Delta }=R^2\Delta $ and $\widetilde{K}=R^2(K
+\Delta \log R)$. Thus, the stability of $(M,g)$ implies that the following
operator is non-negative on $M$:
\[
-\Delta +2K-4H^2=-\frac{1}{R^2}(\widetilde{\Delta }-2\widetilde{K}+q),
\]
where $q=2R^2\Delta \log R+4H^2R^2$. Since $\Delta \log R=\frac{2}{R^4}
(R^2\langle p,\eta \rangle H+\langle p,\eta \rangle ^2)$ where
$\eta $ is the unitary
normal vector field to $M$ (with respect to $g$) for which $H$ is the mean curvature, then
\begin{equation}
\label{eq:parabolaH}
\frac{1}{4}q=H^2R^2+\langle p,\eta \rangle H+\frac{\langle p,\eta \rangle ^2}{R^2}.
\end{equation}
Viewing the right-hand-side of (\ref{eq:parabolaH}) as a quadratic
polynomial in the
variable $H$, its discriminant is $-3\langle p,\eta \rangle ^2\leq
0$. Since the  coefficient of $H^2$ on the right-hand-side of
(\ref{eq:parabolaH}) is $R^2\geq 0$, we deduce that $q\geq 0$ on
$M$. Applying Theorem~2.9 in~\cite{mpr19} to the operator
$\widetilde{\Delta }-2\widetilde{K}+q$ on 
$(M,\wt{g})$, 
we deduce that $(M,\wt{g})$ 
has at most quadratic area growth. This last property implies that
that every bounded solution of the equation $\widetilde{\Delta }u-2\widetilde{K}u+qu=0$
has constant sign on $M$ 
(see Theorem~1 in Manzano, P\'erez and Rodr\'\i guez~\cite{mper1}).

Arguing by contradiction, suppose that $(M,g)$ is not flat. Then,
there exists a bounded Jacobi function $u$ on $(M,g)$ which changes
sign (simply take a point $p\in M$ and choose $u$ as $\langle
\eta ,a\rangle $ where $a\in \R^3$ is a
non-zero tangent vector to $M$ at $p$). Then clearly $u$ satisfies
$\widetilde{\Delta }u-2\widetilde{K}u+qu=0$ on $M$. 
This contradiction proves the lemma.
\end{proof}

\section{Weak $H$-laminations and CMC laminations.}
\label{sec10}

In this section we start by recalling the classical notion of lamination
and discuss some previous results on regularity of these objects when the
leaves have constant mean curvature. Then we will enlarge the class to admit
{\it weak} laminations by allowing certain tangential intersections between
the leaves. These weak laminations and foliations will be studied
in subsequent sections.

\begin{definition}
\label{deflamination}
{\rm
 A {\it codimension-one
lamination} of a Riemannian three-manifold $N$ is the union of a
collection of pairwise disjoint, connected, injectively immersed
surfaces, with a certain local product structure. More precisely, it
is a pair $({\mathcal L},{\mathcal A})$ satisfying:
\begin{enumerate}
\item ${\mathcal L}$ is a closed subset of $N$;
\item ${\mathcal A}=\{ \varphi _{\be }\colon \D \times (0,1)\to
U_{\be }\} _{\be }$ is an atlas of coordinate charts of $N$ (here
$\D $ is the open unit disk in $\R^2$, $(0,1)$ is the open unit
interval in $\R$ and $U_{\be }$ is an open subset of $N$); note that
although $N$ is assumed to be smooth, we
only require that the regularity of the atlas (i.e., that of
its change of coordinates) is of class $C^0$; in other words, ${\cal A}$
is an atlas with respect to the topological structure of $N$.
\item For each $\be $, there exists a closed subset $C_{\be }$ of
$(0,1)$ such that $\varphi _{\be }^{-1}(U_{\be }\cap {\mathcal L})=\D \times
C_{\be}$.
\end{enumerate}

We will simply denote laminations by ${\mathcal L}$, omitting the
charts $\varphi _{\be }$ in ${\mathcal A}$ unless explicitly necessary.
A lamination ${\mathcal L}$ is said to be a {\it foliation of $N$} if ${\mathcal L}=N$.
Every lamination ${\mathcal L}$  decomposes into a
collection of disjoint, connected topological surfaces (locally given by $\varphi
_{\be }(\D \times \{ t\} )$, $t\in C_{\be }$, with the notation
above), called the {\it leaves} of ${\mathcal L}$.
Note that if $\Delta
\subset {\cal L}$ is any collection of leaves of ${\cal L}$, then
the closure of the union of these leaves has the structure of a
lamination within ${\cal L}$, which we will call a {\it
sublamination.}

A codimension-one lamination ${\cal L}$ of $N$ is said to be a {\it CMC lamination}
if each of its leaves is smooth and has constant mean curvature (possibly varying
from leaf to leaf).
Given $H\in \R $, an {\it $H$-lamination} of $N$
is a CMC lamination all whose leaves have the same mean curvature
$H$. If $H=0$, the $H$-lamination is called a {\it minimal
lamination.}
 }
\end{definition}

Since the leaves of a lamination ${\cal L}$ are disjoint, it makes sense to consider the
second fundamental form $\sigma _{\cal L}$ as being defined on the union of the leaves.
A natural question to ask is whether or not the
norm $|\sigma _{\cal L}|$  of the second fundamental form of a
(minimal, $H$- or CMC) lamination ${\cal L}$ in a Riemannian
three-manifold is locally bounded. Concerning this question, we
make the following observations.
\begin{description}
\item[{\rm O.1.}] If ${\cal L}$ is a minimal lamination, then the 1-sided
curvature estimates for minimal disks by Colding and
Minicozzi~\cite{cm23,cm35} imply that $|\sigma _{\cal L}|$
is locally bounded (to prove this, one only needs to deal
with limit leaves, where the 1-sided curvature estimates apply).

\item[{\rm O.2.}]
As a consequence of recent work of Meeks and
Tinaglia~\cite{mt7} on curvature estimates for embedded disks of
positive constant mean curvature and a related 1-sided curvature estimate, a
CMC lamination ${\cal L}$ of a Riemannian three-manifold $N$ has
$|\sigma _{\cal L}|$ locally bounded.
\end{description}

Given a
sequence of CMC laminations ${\cal L}_n$ of a Riemannian
three-manifold  $N$ with uniformly bounded second fundamental form
on compact subsets of $N$, a simple application of the uniform graph lemma
for surfaces with constant mean curvature (see Colding and Minicozzi~\cite{cmCourant}
or P\'erez and Ros~\cite{pro2} from where this well-known result can be deduced)
and of the Arzel\`a-Ascoli Theorem, gives that there exists
a limit object of (a subsequence of) the ${\cal L}_n$,
which in general fails to be a CMC lamination since two ``leaves'' of this limit
object could intersect tangentially with mean curvature vectors pointing in opposite directions;
nevertheless, if ${\cal L}_n$ is a {\it minimal} lamination for every $n$, then
the maximum principle ensures that the limit object is indeed a minimal lamination,
see Proposition~B1 in~\cite{cm23} for a proof. Still, in the general case of
CMC laminations, such a limit object always
satisfies the conditions in the next definition.

\begin{definition}
\label{definition}
{\rm A (codimension-one) {\it weak CMC lamination} ${\cal L}$ of a
Riemannian three-manifold $N$ is a collection
$\{ L_\alpha\}_{\alpha\in I}$ of (not necessarily
injectively) immersed constant mean curvature surfaces
called the {\it leaves} of ${\cal L}$,
satisfying the following three properties.
\begin{enumerate}[1.]
\item $\bigcup_{\alpha\in I}L_{\alpha }$ is a closed
subset of $N$.
\item If $p\in N$ is a point where either two leaves of
${\cal L}$ intersect or a leaf of ${\cal L}$ intersects itself, then
each of these local surfaces at $p$ lies at one side of the other
(this cannot happen if both of the intersecting leaves have the same
signed mean curvature as graphs over their common tangent space at
$p$, by the maximum principle).

\item The function $|\sigma _{\cal L}|\colon {\cal L}\to [0,\infty )$ given by
\begin{equation}
\label{eq:sigma}
|\sigma _{\cal L}|(p)=\sup \{ |\sigma _L|(p)\ | \ L \mbox{ is a leaf of ${\cal L}$ with $p\in L$} \} .
\end{equation}
is uniformly bounded on compact sets of~$N$.
\end{enumerate}
Furthermore:
\begin{itemize}
\item If $N=\bigcup _{\alpha } L_{\a }$, then we call ${\cal
L}$ a {\it weak CMC foliation} of $N$.

\item If the leaves of ${\cal L}$ have the same constant mean curvature
$H$, then we call ${\cal L}$ a {\it weak $H$-lamination} of $N$ (or
$H$-foliation, if additionally $N=\bigcup _{\alpha } L_{\a }$).
\end{itemize}
}
\end{definition}

\begin{remark}
   {\em
   \begin{description}
   \item[]
   \item[{\rm 1.}] The function $|\sigma _{\cal L}|$ defined in (\ref{eq:sigma}) for a CMC lamination is not
   necessarily continuous, as demonstrated by the following
example: consider the union in $\R^3$ of $\Pi _0=\{ (x_1,x_2,x_3)\ | \ x_3=0\} $ and the sphere
$\esf ^2(p_1,1)$ where $p_1=(0,0,1)$. Also note that
this example can be modified to create a weak CMC foliation ${\cal F}$ of $\R^3$ minus a point with
non-continuous related function $|\sigma _{\cal F}|$: add to the previous example ${\cal L}$
all planes $\Pi _t=\{ x_1,x_2,t)\ | \ t<0\} $, and foliate the open set $\{ x_3>0\} -\{ (0,0,1)\} $ by
the spheres $\esf^2(p_t,t)$ with $t\geq 1$ where $p_t=(0,0,t)$,
together with the spheres $\esf^2(p_1,r)$, $r\in (0,1)$. In either case, $|\sigma _{\cal L}|$
or $|\sigma _{\cal F}|$ is not continuous at the origin.
 \item[{\rm 2.}] A weak $H$-lamination for $H=0$ is a minimal lamination in the
sense of Definition~\ref{deflamination}.
\item[{\rm 3.}] As a consequence of Observation~O.2 above, every CMC lamination
(resp. CMC foliation) of $N$ is a weak CMC lamination (resp. weak CMC foliation).
  \end{description}
  }
\end{remark}

The following proposition follows immediately from the definition of
a weak $H$-lamination and the maximum principle for $H$-surfaces.

\begin{proposition}
\label{prop10.2}
Any weak $H$-lamination
${\cal L}$ of a three-manifold $N$ has a local $H$-lamination
structure on the mean convex side of each leaf.  More precisely,
given a leaf $L_{\a }$ of ${\cal L}$ and given a small disk $\Delta
\subset L_{\alpha }$, there exists an $\ve >0$ such that if $(q,t)$
denotes the normal coordinates for $\exp _q(t\eta _q)$ (here $\exp $
is the exponential map of $N$ and $\eta $ is the unit normal vector
field to $L_{\a }$ pointing to the mean convex side of $L_{\a }$),
then the exponential map $\exp $ is an injective submersion in
$U(\Delta ,\ve ):= \{ (q,t) \ | \ q\in \mbox{\rm Int}(\Delta ), \, t\in
(-\ve ,\ve )\} $, and the inverse image $\exp^{-1}({\cal L})\cap \{
q\in \mbox{\rm Int}(\Delta ), \,t\in [0,\ve )\} $ is an $H$-lamination of
$U(\Delta ,\ve $) in the pulled back metric, see
Figure~\ref{figHlamin}.
\begin{figure}
\begin{center}
\includegraphics[width=5.1cm,height=4cm]{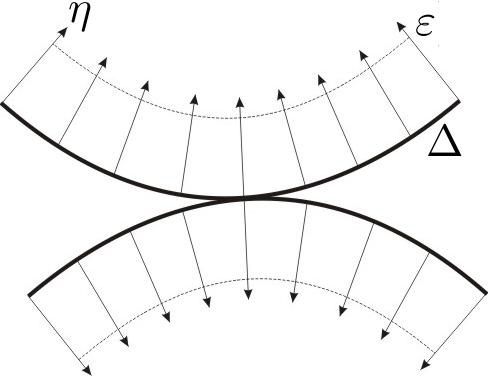}
\caption{The leaves of a weak $H$-lamination with $H\neq 0$ can
intersect each other or themselves, but only tangentially
with opposite mean curvature vectors. Nevertheless,
on the mean convex side of these locally intersecting leaves,
there is a lamination structure.
}
\label{figHlamin}
\end{center}
\end{figure}
\end{proposition}

\begin{definition}
\label{def-limset}
 {\rm
Let $M$ be a complete, embedded surface in a Riemannian three-manifold $N$. A
point $p\in N$ is a {\it limit point} of $M$ if there exists a
sequence $\{p_n\}_n\subset M$ which diverges to infinity in $M$ with
respect to the intrinsic Riemannian topology on $M$ but converges in
$N$ to $p$ as $n\to \infty$. Let lim$(M)$ denote the set of all limit
points of $M$ in $N$; we call this set the {\it limit set of $M$}.
In particular, lim$(M)$ is a closed subset of $N$ and $\overline{M} -M
\subset \lim (M)$, where $\overline{M}$ denotes the closure of~$M$.

The above notion of limit point can be extended to the case of
a lamination ${\cal L}$ of $N$ as follows:
A point $p\in \mathcal{L}$ is a {\it limit point} if there exists a coordinate chart
$\varphi _{\beta }\colon \D \times (0,1)\to U_{\be }$ as in Definition~\ref{deflamination}
such that $p\in U_{\be }$ and $\varphi _{\be }^{-1}(p)=(x,t)$ with
$t$ belonging to the accumulation set of $C_{\be }$.
The notion of limit point can be also extended to the case of a weak $H$-lamination of $N$,
by using that such an weak $H$-lamination has a local lamination structure at the mean convex side
of any of its points, given by Proposition~\ref{prop10.2}.
It is easy to show that if $p$ is a limit point of
a lamination ${\cal L}$ (resp. of a weak $H$-lamination), then the leaf $L$ of ${\cal L}$
passing through $p$ consists entirely of limit points of ${\cal L}$;
in this case, $L$ is called a {\it limit leaf} of ${\cal L}$.}
\end{definition}

\section{Proof of Theorem~\ref{thm5.7}.}
\label{sec4}

We now prove the universal curvature estimates for weak CMC foliations stated in
Theorem~\ref{thm5.7}. Let $N$ be a compact Riemannian three-manifold possibly with boundary,
whose absolute sectional curvature is at most $\Lambda \geq 0$. Let ${\cal F}$ be
a weak CMC foliation of $N$ and $p\in \rm Int(N)$. Recall that $|\sigma _{\cal F}|$ is the function
defined in (\ref{eq:sigma}).

Note that $|\sigma_{\cal F}|(p)\min\{{\rm dist}_N(p,\partial N), \frac{\pi }{\sqrt{\Lambda }}\}$ is invariant
under rescaling of the ambient metric. This invariance implies that we can fix $\Lambda
>0$ and prove that $|\sigma_{\cal F}|(p)\min\{{\rm dist}_N(p,\partial N), \frac{\pi }{\sqrt{\Lambda }}\}$
is bounded independently of the compact Riemannian three-manifold $N$ with boundary
whose absolute sectional curvature is at most $\Lambda $ and independently of the weak CMC foliation ${\cal F}$ of $N$.

We fix $\Lambda >0$. Arguing by contradiction, assume that there exists a sequence of
weak CMC foliations ${\cal F}_n$ of compact three-manifolds $N_n$ with
boundary, such that the absolute sectional curvature of $N_n$ is at
most $\Lambda $, and there exists a sequence of points $p_n$ in leaves
$L_n$ of ${\cal F}_n$ with
$|\sigma _{L_n}|(p_n)\min \{ r_n,\frac{\pi }{\sqrt{\Lambda}}\} \geq n$
for all $n$, where  $r_n=d_{N_n}(p_n,\partial N_n)$.
After replacing $r_n$ by $\min \{ r_n,\frac{\pi }{\sqrt{\Lambda}}\} $
and $N_n$ by the closed metric ball $\overline{B}_{N_n}(p_n,r_n)$,
we can assume $r_n\leq \frac{\pi }{\sqrt{\Lambda }}$.

Since the sectional curvature of $N_n$ is at most $\Lambda $, a standard comparison argument
for zeros of Jacobi fields on geodesics gives that the exponential map
\[
\exp _{p_n}\colon B(\vec{0},r_n)\subset T_{p_n}N_n\to B_{N_n}(p_n,r_n)
\]
is a local diffeomorphism from the ball of radius $r_n$ centered at the origin in the
tangent space to $N_n$ at $p_n$ (endowed with the ambient metric at $p_n$) onto the
metric ball $B_{N_n}(p_n,r_n)$. After lifting the ambient metric of $N_n$ to
$B(\vec{0},r_n)$, we can consider the above map to be a local isometry (note that
the metric on $B(\vec{0},r_n)$ depends on $n$). We can also consider the weak CMC foliation
${\cal F}_n$ to be a weak CMC foliation of $B(\vec{0},r_n)$ with the pulled back metric.
Next consider the homothetic expansion of the metric in $B(\vec{0},r_n)$ centered at
the origin with ratio $\frac{\pi }{r_n\sqrt{\Lambda }}$. 
After this new normalization, we have the following properties:
\begin{description}
\item[(P1)] $B(\vec{0},r_n)$ becomes a metric ball of radius $\frac{\pi }{\sqrt{\Lambda }}$,
which we will denote by $B_n(\vec{0},\frac{\pi }{\sqrt{\Lambda }})$, since its
Riemannian metric still depends on $n$.

\item[(P2)] The absolute sectional curvature of
$B_n(\vec{0},\frac{\pi }{\sqrt{\Lambda }})$ is less than or equal to $\Lambda $.

\item[(P3)] We have a related CMC foliation ${\cal F}'_n$ on $B_n(\vec{0},\frac{\pi }{\sqrt{\Lambda }})$
so that the second fundamental form of some leaf $L_n'$ of ${\cal F}'_n$ passing through $\vec{0}$ satisfies
$|\sigma _{L'_n}|(\vec{0})\frac{\pi }{\sqrt{\Lambda }}\geq n$ for all $n\in \N$.
\end{description}

By Lemma~2.2 in~\cite{rst1}, the injectivity radius function $I_n$ of $B_n(\vec{0},\frac{\pi }{\sqrt{\Lambda }})$
satisfies $I_n(x)\geq \frac{\pi }{4\sqrt{\Lambda }}$ for all points $x\in B_n(\vec{0},
\frac{\pi }{4\sqrt{\Lambda }})$. 
By Theorem~2.1 in~\cite{rst1}, given $\a \in (0,1)$ there exists $r_0>0$ (only depending on
$\Lambda $ but not on $n$) so that we can pick harmonic coordinates
in the metric ball $B_n(x,r_0)$ centered at any point $x\in B_n(\vec{0},\frac{\pi }{8\sqrt{\Lambda }})$
of radius $r_0$. 
For the notion of harmonic coordinates, see~\cite{rst1}; the only property
we will use here about these harmonic coordinates is that the metric tensor on $B_n(x,r_0)$
(which depends on $n$) is $C^{1,\a }$-controlled.

Fix $n\in \N$ and let $q_n\in B_n(\vec{0},\frac{\pi }{8\sqrt{\Lambda }})$ be a supremum of
the function
\begin{equation}
\label{eq:fn}
q\in {\textstyle B_n(\vec{0},\frac{\pi }{8\sqrt{\Lambda }})}
\mapsto f_n(q)=
|\sigma _{{\cal F}'_n}|(q)\,  d_n(q,\partial {\textstyle B_n(\vec{0},\frac{\pi }{8\sqrt{\Lambda }})}),
\end{equation}
where given $q\in B_n(\vec{0},\frac{\pi}{8\sqrt{\Lambda}})$,
 \, $|\sigma _{{\cal F}'_n}|(q)=\sup \{ |\sigma _{L'}|(q) \ | \ L'\in {\cal F}'_n, \ q\in L'\}  $
and $d_n$ denotes extrinsic distance in $B_n(\vec{0},\frac{\pi }{8\sqrt{\Lambda }})$.
We observe that the following properties hold:
\begin{itemize}
\item The supremum in (\ref{eq:fn}) exists since the second fundamental form of all leaves
of ${\cal F}'_n$ is uniformly bounded in $B_n(\vec{0},\frac{\pi }{8\sqrt{\Lambda }})$
by definition of weak CMC foliation.

\item $f_n$ may not be continuous (as $q\mapsto |\sigma _{{\cal F}'_n}|(q)$ might fail to be
continuous) but still it is bounded, and it vanishes at
$\partial B_n(\vec{0},\frac{\pi }{8\sqrt{\Lambda }})$. A similar argument as in the observation above
shows that the supremum of $f_n$ is attained at an interior point of $B_n(\vec{0},\frac{\pi }{8\sqrt{\Lambda }})$.

\item The value of $f_n$ at $\vec{0}$ tends to $\infty $ as $n\to \infty $.
\end{itemize}

Let $s_n=\min \{ r_0,
\frac{1}{2}d_n(q_n,\partial B_n(\vec{0},
\frac{\pi }{8\sqrt{\Lambda }}))\} $ and let
$\lambda _n=|\sigma _{{\cal F}'_n}|(q_n)$. After rescaling the
metric of the ball $B_n(q_n,s_n)$ of radius $s_n$ centered at $q_n$
by the factor $\l _n$, we have associated weak CMC foliations ${\cal F}''_n$
of $\lambda_n\, B_n(q_n, s_n)$ such that the norm of the second fundamental
form of ${\cal F}''_n$ is at most $2$ everywhere and is equal to 1 at
the center $q_n$ of this ball.

Using the techniques described in~\cite{rst1}, it follows that a
subsequence of the ${\cal F}''_n$ converges to a weak CMC foliation
${\cal Z}$ of $\rth$; next we sketch an explanation of these techniques:
First, one uses the above harmonic coordinates to show that
the coordinatized Riemannian manifolds $\l _n\, B_n(q_n,s_n)$ converge uniformly on compact
subsets of $\R^3$ in the $C^{1,a}$-Euclidean topology to $\R^3$ endowed with
its Euclidean metric. Then one shows that the leaves in ${\cal F}''_n$
can be locally written as graphs of functions defined over Euclidean disks of uniformly
controlled size (this property follows from the uniform graph lemma for surfaces with constant mean
curvature, note that for this we need an uniform bound on the second fundamental
form of the leaves of ${\cal F}''_n$, which is obtained by the blow-up process).
Another consequence of the uniform graph lemma is that one obtains uniform
local $C^2$-bounds for the graphing functions. The next step consists of using that the graphing functions
satisfy the corresponding mean curvature equation, which is an elliptic PDE of
second order whose coefficients have uniform $C^{0,\a}$-estimates (this follows
from the $C^2$-bounds for the graphing functions and the $C^{1,\a }$-control
of the ambient metric on $\l _n\, B_n(q_n,s_n)$, see Lemma~2.4 in~\cite{rst1}),
together with Schauder estimates to conclude local uniform $C^{2,\a }$-bounds
for the graphing functions of the leaves of ${\cal F}''_n$. Finally, these
local uniform $C^{2,\a }$-bounds for the graphing functions allow us to use the Arzel\`a-Ascoli
Theorem to obtain convergence (after extracting a subsequence) in the $C^2$-topology to
limit graphing functions of class $C^2$. Using that the convergence is $C^2$ one can pass
to the limit the mean curvature equations satisfied by the graphing functions in the
sequence, thereby producing a  (local) limit weak CMC foliation of an open set of $\R^3$.
Finally, a diagonal argument produces a global limit weak CMC foliation
${\cal Z}$ of $\rth$ of a subsequence of the ${\cal F}''_n$. For details, see~\cite{rst1}.

The above process insures that the limit weak  CMC foliation
${\cal Z}$ of $\rth$ satisfies the following properties:
\begin{enumerate}
\item The second fundamental form of the leaves of
${\cal Z}$ is bounded in absolute value by $1$  (in particular,
there is a bound on the mean curvature of every leaf of ${\cal Z}$),
and there is a leaf $\Sigma $ of ${\cal Z}$ passing through the origin
which is not flat.
\item ${\cal Z}$ is not a minimal foliation (otherwise
${\cal Z}$ would consist entirely of planes, contradicting
the existence of $\Sigma $).
\end{enumerate}

Since the leaves of ${\cal Z}$ have uniformly bounded second
fundamental forms, after a sequence of translations of ${\cal Z}$ in
$\R^3$, we obtain another limit weak CMC foliation $\wh{\cal Z}$ of
$\rth$ with a leaf $\wh{L}$ passing through the origin which has non-zero
{\it maximal} mean curvature among all leaves of $\wh{\cal Z}$.
But the two-sided surface $\wh{L}$ is then stable by
Proposition~5.4 in~\cite{mpr19} and since it is also complete, then
$\wh{L}$ must be flat (see for instance Lemma~\ref{lema1} for a proof of
this well-known result). This contradiction finishes the proof of the theorem.
{\hfill\penalty10000\raisebox{-.09em}{$\Box$}\par\medskip}

In Section~\ref{sectspheres} we will use the next corollary, which follows immediately from
Theorem~\ref{thm5.7} after  scaling the estimate
for balls of radius 1.

\begin{corollary} \label{cor:scale}
There exists an $A>0$ such that if  ${\cal F}$ is a  weak $CMC$ foliation
 of a ball $\B(p,R)\subset \rth$, then $|\sigma_{\cal F}|(p)\leq A/R$,
 where $|\sigma _{\cal F}|$ is given by (\ref{eq:sigma}).
\end{corollary}

\section{Proof of Theorem~\ref{tt2}.}
\label{secprooftt2}

Let ${\cal L}$ be a weak $H$-lamination of a
punctured ball $B_N(p,r)$ in a Riemannian three-manifold $N$,
such that $|\sigma _{\cal L}|\, d_N(p,\cdot )\leq C$ for some $C>0$.
To prove Theorem~\ref{tt2}, it suffices to check that ${\cal L}$
extends to a weak $H$-lamination
of $B_N(p,r)$ for a smaller $r>0$. Throughout this section,
we will assume without loss of generality that $r$ is sufficiently small
so that the exponential map $\exp _p$ restricted
to $B(\vec{0}, r)\subset T_pN=\rth$ induces $\rth$-coordinates
on $B_N(p,r)$.

\begin{lemma}
\label{ass:flat}
Let ${\cal L}$ be a weak $H$-lamination of a
punctured ball $B_N(p,r)$ in a Riemannian three-manifold $N$,
such that $|\sigma _{\cal L}|\, d_N(p,\cdot )\leq C$ for some $C>0$.
Then for every sequence of positive numbers $C_n\searrow 0$, there
exists another sequence $r_n\searrow 0$ such that
\[
|\sigma_{\cal L}|\, d_N(p,\cdot )\leq C_n \quad \mbox{in } \,\, {\cal L}\cap B_N(p,r_n).
\]
\end{lemma}
\begin{proof}
Arguing by contradiction, suppose that the lemma fails. It follows that
there exists an $\ve>0$
and a sequence of points $p_n\in {\cal L}$ such that
$\lim_{n\to \infty} p_n=p$ and
$\ve\leq |\sigma_{\cal L}|(p_n) \, d_N (p,p_n)$ for each $n$.
Let $\lambda_n =\frac{1}{d_N (p,p_n)}$ and consider
the sequence of rescaled weak $\frac{H}{\lambda _n}$-laminations
${\cal L}_n=\lambda_n {\cal L}\subset \lambda_n B_N(p,r)$;
here by $\l _nB_N(p,r)$ we mean $B_N(p,r)$ endowed with the
Riemannian metric $\l _n^2\langle ,\rangle $, where $\langle ,
\rangle $ is the metric on $N$.

As $H$ is fixed, $\l _n\to \infty $, $|\sigma _{\cal L}|\, d_N(p,\cdot )\leq C$ and
$|\sigma _{\cal L}|\, d_N(p,\cdot )$ is invariant under
homothetic rescalings of the metric around $p$, then there exists a subsequence
of the weak $\frac{H}{\l _n}$-laminations ${\cal L}_n$ that converges to a
minimal lamination ${\cal L}'$ of $\rth-\{ \vec{0}\} $,
which furthermore satisfies $|\sigma_{{\cal L}'}|\, R\leq C$
in $\R^3-\{ \vec{0}\} $ (recall that $R$ stands for radial
distance in $\R^3$ to the origin). As $\ve
\leq |\sigma_{\cal L}|(p_n) \, d_N (p,p_n)$ for all $n$,
then $\ve \leq |\sigma_{{\cal L}'}|(q_{\infty })$ for
some $q_{\infty }\in {\cal L}'\cap \esf ^2(1)$; thus the lamination
${\cal L}'$ is not flat. In this setting,
Corollary~6.3 in~\cite{mpr10} gives that the closure $\overline{{\cal L}'}$
of ${\cal L}'$ in $\R^3$ consists of a single leaf $\overline{L'}$,
which is a non-flat minimal surface
with finite total curvature (clearly ${\cal L}'=\{ L'\} $ where
$L'=\overline{L'}-\{ \vec{0}\} $).
\par
\vspace{.2cm}
\noindent
{\bf Claim A:} {\it $\overline{L'}$ contains the origin $\vec{0}$.}
\par
\vspace{.2cm}
\noindent
To see this, it suffices to show that the distance sphere
$S_N^2(p,\de )$ intersects ${\cal L}$ for every $\de >0$ sufficiently small.
Otherwise, there exists $\de _1 >0$ such that the following properties hold:
\begin{enumerate}
  \item $S^2_N(p,\de _1)\cap {\cal L}=\mbox{\O }$.
  \item For all $\de \in (0,\de _1]$, the mean curvature function of
  $S^2_N(p,\de )$ is strictly greater than $H$.
  \item The family $\{ S^2_N(p,\de )\ | \ \de \in (0,\de _1]\} $ foliates the punctured closed ball
  $\overline{B}_N(p,\de _1)-\{ p\} $.
 \end{enumerate}
Since $B_N(p,\de _1)$ intersects ${\cal L}$ and ${\cal L}$ is a closed subset of
$\overline{B}_N(p,\de _1)-\{ p\} $, then there exists a largest $\de _2\in (0,\de _1)$ such that $S^2_N(p,\de _2)\cap {\cal L}\neq \mbox{\O }$. This contradicts the mean curvature comparison principle,
which proves Claim A.
\par
\vspace{.2cm}
\noindent
{\bf Claim B:} {\it There exists $r'\in (0,r)$ such that the following properties
hold:
\begin{description}
\item[(B1)] The intersection of ${\cal L}$ with $\overline{B}_N(p,r')$
consists of a single leaf $L$ of the
induced lamination. Furthermore, $L\cap S^2_N(p,r')$ is a simple closed curve
    along which $L$ and $S^2_N(p,r')$ intersect almost orthogonally.
\item[(B2)] $L\cap \overline{B}_N(p,r')$ is properly embedded in $\overline{B}_N(p,r')-\{ p\} $, with $p$ in its closure.
\end{description}
}
\par
\vspace{.2cm}
\noindent
To see that Claim B holds, first note that by Claim A, the intersection of $L'$
with any closed ball $\overline{\B }(R)$ centered at $\vec{0}$ of
sufficiently small radius $R>0$ is a punctured disk which is almost orthogonal to $\esf^2 (R)$.
Since $L'$ is not flat, then the  convergence of the laminations ${\cal L}_n $ to $L'$ has multiplicity
one (by Lemma~4.2 in~\cite{mpr20}).
Thus, we deduce that for $n$ large and $r'_n:=d_N(p_n,p)R$, there
exists a unique leaf $L(n)$ of ${\cal L}$ that intersects
$S^2_N(p,r'_n)$, and this intersection is a simple closed curve along which
$L(n)$ and $S^2_N(p,r'_n)$ intersect almost orthogonally.
If $L(n)$ were not the unique leaf of ${\cal L}$ that intersects $B_N(p,r'_n)$,
then ${\cal L}\cap B_N(p,r'_n)$ would contain a non-empty
sublamination which does not intersect $S^2_N(p,r'_n)$.
A similar comparison argument for the mean curvature as in the proof of Claim A shows that
this is impossible for $n$ sufficiently large. Hence item~{\bf (B1)} above holds by taking
$r'=r'_n$ and $L=L(n)$ for $n$ large.

Suppose that $L$ is a limit leaf of ${\cal L}$. In this case, item~{\bf (B1)} implies that
every point in $L\cap S^2_N(p,r')$ is the limit in $N$ of a
sequence of points of $L$ itself. This is impossible, since $L$
and $S^2_N(p,r')$ intersect almost orthogonally in a Jordan curve.
Thus, $L$ is not a limit leaf of ${\cal L}$. We next prove item~{\bf (B2)}:
If $L$ were not proper in $B_N(p,r')-\{ p\} $, then ${\cal L}\cap B_N(p,r')$ would contain a limit leaf, which
therefore would not be $L$; this contradicts {\bf (B1)}. Finally, if $p$ is not in the closure of $L$,
then $p$ is at positive distance from $L$. This contradicts {\bf (B1)} together with
$\ve \leq |\sigma _{\cal L}|(p_n)d(p,p_n)$ as $p_n$ converges to $p$. Now {\bf (B2)} is proved, as well
as Claim B.
\par
\vspace{.2cm}
\noindent
We next finish the proof of the lemma.
Since $L$ is properly embedded in $B_N(p,r')-\{ p\} $, then $L$ is a locally rectifiable current in
$B _N(p,r')-\{ p\} $. As $L$ has bounded mean curvature (actually constant), then
Theorem 3.1 in Harvey and Lawson~\cite{hl1} implies that $L$ has finite area. Since $L$ has bounded mean
curvature and finite area, the monotonicity formula in Corollary~5.3 of Allard~\cite{al1}
implies that $L$ has a well-defined finite density at $p$. In this setting, we can apply Theorem 6.5
in~\cite{al1} to deduce that under any sequence of homothetic expansions $\{ L'_n\}_n$ of $L$,
the surfaces $L'_n$ converge (up to a subsequence) to a cone ${\cal C}_{L}\subset \R^3$
(depending on the sequence), which is the cone over a stationary, integral one-dimensional varifold $\G $
in the unit two-sphere of $\R^3$, and ${\cal C}_L$ is flat at its smooth points.
But the blow-up limit $L'$ is smooth and not flat, which is a contradiction.
This contradiction proves the lemma.
\end{proof}

By Lemma~\ref{ass:flat}, it follows that the following property holds:
\par
\vspace{.2cm}
\noindent
{\bf (P)} {\it Under rescaling by every sequence
$\{\lambda_n\}_n\subset (0,\infty )$ with $\lambda_n\to \infty$, a
subsequence of the weak $\frac{H}{\l _n}$-laminations
${\cal L}_n =\lambda _n{\cal L}\subset \lambda_n B_N(p,r)$
converges in $\R^3-\{ \vec{0}\} $ to a lamination ${\cal L}'$ of
$\rth$ by parallel planes. (Note that  ${\cal L}'$ might depend on
$\{\l_n\}_n$).
}
\par
\vspace{.2cm}
\noindent

\begin{proposition}
\label{propos5.2}
Theorem~\ref{tt2} holds in the particular case $N=\rth$ and $p=\vec{0}$.
\end{proposition}
{\it Proof.} As property {\bf (P)} holds, it follows that for
$\ve > 0$ sufficiently small, in the annular domain
$A = \{ x \in \rth \mid \frac{1}{2} \leq |x| \leq 2 \}$,
the normal vectors to the leaves of $(\frac{1}{\ve} {\cal L}) \cap A$ are almost parallel,
and after a rotation (which might depend on $\ve $), we will assume that the unit normal
vector to the leaves of $(\frac{1}{\ve} {\cal L}) \cap A$ lies in a small
neighborhood of $\{ \pm (0,0,1)\} $.
Hence, for such a sufficiently small $\ve$, each
 component $C$ of $(\frac{1}{\ve} {\cal L}) \cap A$ that intersects
${\esf}^2(1)$ is of one of the following four types
, see
Figure~\ref{fig4new}:
\begin{figure}
\begin{center}
\includegraphics[width=14cm,height=6.7cm]{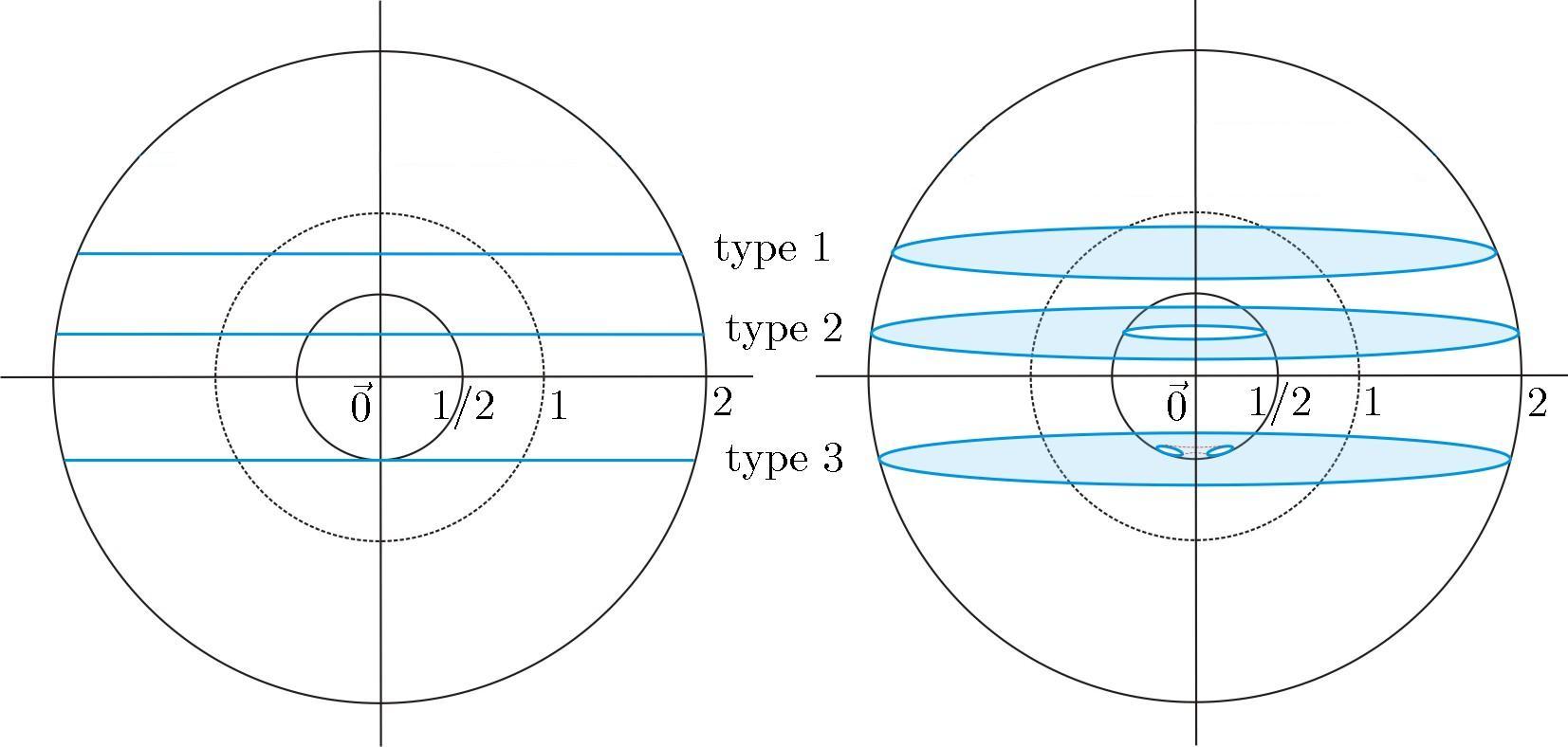}
\caption{Type 1, 2, 3 connected components of ${\cal L}_{\ve }$.}
\label{fig4new}
\end{center}
\end{figure}
\begin{description}
\item[{\rm Type 1.}] $C$ is a compact disk with boundary $\G (C)$ in ${\esf}^2(2)$.
\item[{\rm Type 2.}] $C$ is a compact annulus with one boundary curve $\G (C)$ in
${\esf}^2 (2)$ and the other boundary curve in ${\esf}^2 (\frac{1}{2})$.
\item[{\rm Type 3.}] $C$ is a compact planar domain whose boundary
 consists of a single closed curve $\G (C)$  in
 ${\esf}^2 (2)$ together with at least two closed
 curves in ${\esf}^2 (\frac{1}{2})$,  and
where $\G (C)$ bounds a compact disk in
$\frac{1}{\ve} {\cal L}$;
\item[{\rm Type 4.}] $C$ is an infinite multigraph whose limit set consists
of two compact components of $(\frac{1}{\ve} {\cal L}) \cap A$ of type~2
(any such spiraling component does not intersect the intersection
of $A$ with an open slab of small width around height $\pm \frac{1}{2}$).
\end{description}

We also define
$A(n)=\{ x \in \rth \mid \frac{1}{2^{2n+1}} \leq |x| \leq \frac{1}{2^{2n-1}} \} $
for each $n\in \N \cup \{ 0\} $ (so $A=A(0)$). Note that $\bigcup _{n\in \N\cup \{ 0\} }
A(n)=\overline{\B}(2)-\{ \vec{0}\} $. Given a component
$C$ of $(\frac{1}{\ve} {\cal L}) \cap A$, let
$\Delta _C$ be the leaf of $(\frac{1}{\ve }{\cal L})\cap [\overline{\B }(2)-\{
\vec{0}\} ]$ that contains $C$. Given $n\in \N\cup \{ 0\} $ fixed, the above
division of components $C$ of $(\frac{1}{\ve} {\cal L}) \cap A(0)$ can be
directly adapted to components $\Delta _C(n)=\Delta _C\cap A(n)$ of
$(\frac{1}{\ve} {\cal L}) \cap A(n)$. We make the following elementary
observations:
\begin{description}
\item[(O1)] If for some $n\in \N \cup \{ 0\} $, $\Delta _C(n)$ is of type 4, then
$\Delta _C(n')$ is of type 4 for every $n'\in \N \cup \{ 0\} $, and $\Delta _C$
has $\vec{0}$ in its closure.
\item[(O2)] If for some $n\in \N \cup \{ 0\} $, $\Delta _C(n)$ is either empty,
of type 1 or of type 3, then $\Delta _C$ is a disk which is at positive distance from $\vec{0}$.
\item[(O3)] If neither {\bf (O1)} nor {\bf (O2)} occur, then $\Delta _C(n)$
is of type 2 for every $n$ and thus, $\Delta _C$ is a proper annulus limiting to
$\vec{0}$.
\end{description}

\begin{assertion}
\label{ass5.3}
There exists a leaf of $\frac{1}{\ve}{\cal L}$ having $\vec{0}$
in its closure.
\end{assertion}
\begin{proof}
If $(\frac{1}{\ve} {\cal L}) \cap A$ contains a component of
type 4, then {\bf (O1)} implies that our claim holds. If $(\frac{1}{\ve} {\cal L}) \cap A$
contains a component $C$ of type 2 such that $\Delta _C(n)$ is of type 2 for every $n$,
then {\bf (O3)} insures that $\Delta _C$ contains $\vec{0}$ in its closure
and we are also done. We will prove that the remaining case is impossible
and this will finish the proof of this assertion.

The remaining case is that
for every component $C$ of $(\frac{1}{\ve} {\cal L}) \cap A$, there exists
$n\in \N\cup \{ 0\} $ such that $\Delta _C(n)$ is empty, of type 1 or of type 3; in the
case that $\Delta _C(n)$ is of type 3, then $\Delta _C(n')$ is empty
for every $n'>n+1$. By {\bf (O2)}, we have that $(\frac{1}{\ve} {\cal L}) \cap [\overline{\B }(2)-\{ \vec{0}\} ]$
consists of an (infinite) collection of pairwise disjoint compact disks.
Since $\vec{0}$ is in the closure of ${\cal L}$, there exists a sequence of points $\{
p_m\} _{m\in \N }$
in compact disk leaves $D(p_m)$ of $(\frac{1}{\ve} {\cal L}) \cap [\overline{\B }(2)-\{ \vec{0}\} ]$,
such that the $p_m$ converge to $\vec{0}$ as $m\to \infty $. We define for every $k\in \N$,
\[
{\cal D}(k)=\bigcup _{m=k}^{\infty }D(p_m).
\]

We claim that there exists $k\in \N$ such that ${\cal D}(k)$ is not closed in
$\overline{\B }(2)-\{ \vec{0}\} $. Otherwise, $\{ {\cal D}(k)\cap A\ | \ k\in
\N \} $ is a collection of closed subsets of the compact space $A$, which
clearly satisfies the finite intersection property; therefore, there exists a point
$q\in [\cap _{k=1}^{\infty }{\cal D}(k)]\cap A$. In particular, $q\in D(p_j)$
for some $j\in \N$. But as $D(p_j)$ is disjoint from ${\cal D}(j+1)$, we
arrive to a contradiction. This contradiction proves that
there exists $k\in \N$ such that ${\cal D}(k)$ is not closed in
$\overline{\B }(2)-\{ \vec{0}\} $.

Since ${\cal D}(k)$ is not closed in $\overline{\B }(2)-\{ \vec{0}\} $,
then there exists a point $x\in \overline{\B }(2)-\{ \vec{0}\} $ which is
in the closure of ${\cal D}(k)$ but not in ${\cal D}(k)$. As $\frac{1}{\ve }
{\cal L}$ is closed in $\overline{\B }(2)-\{ \vec{0}\} $ and
${\cal D}(k)\subset \frac{1}{\ve }{\cal L}$, then $x\in \frac{1}{\ve }
{\cal L}$. Thus, there exists a leaf component $D_x$
of $\frac{1}{\ve }{\cal L}$ passing through $x$, which is disjoint from
${\cal D}(k)$ as $x\in D_x-{\cal D}(k)$. Furthermore, by our previous arguments,
$D_x$ is topologically a closed disk which is at positive distance
from $\vec{0}$. Choose a compact neighborhood $U$ of $D_x$ in $\overline{\B }(2)-
\{ \vec{0}\} $ which does not contain the origin. Since $\frac{1}{\ve }{\cal L}$ is a
lamination and $D_x$ is a closed disk leaf, $U$ can be chosen so that every leaf of
$\frac{1}{\ve }{\cal L}$ which intersects $U$ is entirely contained in $U$.
If follows that there exists a subsequence of the disks $\{ D(p_m)\} _m$
which is contained in $U$. This is clearly a contradiction,
as the $D(p_m)$ contain points that converge to $\vec{0}$. This proves Assertion~\ref{ass5.3}.
\end{proof}

\begin{assertion}
\label{ass5.4}
There exists a leaf $L$ of $\frac{1}{\ve}{\cal L}$ which is a proper annulus having $\vec{0}$
in its closure. Moreover, every such $L$ extends smoothly across $\vec{0}$.
\end{assertion}
\begin{proof}
By Assertion~\ref{ass5.3}, there exists a component $C$ of $(\frac{1}{\ve}{\cal L})\cap A$
such that $\Delta _C$ has $\vec{0}$ in its closure. By observation {\bf (O2)} above,
$\Delta _C(n)$ is of type 2 or 4 for all $n\in \N$. If $\Delta _C(n)$ is of type~2 for
some $n$ (hence for all $n$), then $\Delta _C$ is a proper annulus limiting to
$\vec{0}$. Otherwise, $\Delta _C(n)$ is of type 4
for some $n$ (hence for all $n$), and thus the limit set of $C$ in $A$ produces two compact
components $C_1,C_2$ of $(\frac{1}{\ve}{\cal L})\cap A$ each of which is of type 2 and
such that $\Delta _{C_1},\Delta _{C_2}$ are proper annuli limiting to $\vec{0}$.
This proves the first sentence of the assertion.

We next choose a proper annular leaf $L$ of $\frac{1}{\ve}{\cal L}$
with $\vec{0}\in \overline{L}$ and check that $L$ extends
smoothly across $\vec{0}$. Since every blow-up limit of $L$ is a
lamination of $\rth$ by parallel planes (property {\bf (P)}), then  $L$ intersects
small spheres $\esf ^2(r')$ of radius $0<r'\ll r$ almost orthogonally in a
curve of length no greater than $3\pi r'$, and thus, $L$ has finite area.
The conformal structure of $L$ must be the one of a punctured disk,
as follows from the fact that under the conformal change of metric $\widetilde{g}=
\frac{1}{R^2}\langle ,\rangle $, $(L,\widetilde{g}|_L)$ has linear area growth
(here, $R=\sqrt{x_1^2+x_2^2+x_3^2}$ and $\langle ,\rangle $ is the inner product
in $\R^3$, recall that our present goal is to prove Proposition~\ref{propos5.2},
which is the $\R^3$-case of Theorem~\ref{tt2}). Since
$L$ has finite area and is conformally a punctured disk,
then $L$ can be conformally parameterized
by a mapping from a punctured disk into $\R^3$ with
finite energy. In this setting, the main theorem
in~\cite{gr1} (which holds true even if we exchange our current
ambient manifold $\R^3$
by any Riemannian three-manifold whose sectional
curvature is bounded from above and whose injectivity
radius is bounded away from zero, conditions which
will be satisfied in the general setting for $N$ that will be dealt with in
Proposition~\ref{propos5.7} below, since we work in
an arbitrarily small ball $B_N(p,r')$, $r'\in (0,r)$),
implies that $L$ extends $C^1$ through $p$
and so, standard elliptic theory gives that $L$ extends
smoothly across $\vec{0}$ as a mapping. Since $L$ is embedded
around $\vec{0}$, then the extended image surface is also smooth.
This completes the proof of Assertion~\ref{ass5.4}.
\end{proof}

\begin{assertion}
\label{ass5.5}
There are no type 4 components of $(\frac{1}{\ve }{\cal L})\cap A$.
\end{assertion}
\begin{proof}
Arguing by contradiction, if there exists a type 4 component
$C$ of $(\frac{1}{\ve }{\cal L})\cap A$, then
the limit set of $C$ in $A$ produces two compact
components $C_1,C_2$ of $(\frac{1}{\ve}{\cal L})\cap A$ each of which is of type 2 and
such that $\Delta _{C_1},\Delta _{C_2}$ are proper annuli limiting to $\vec{0}$.
By Assertion~\ref{ass5.4}, both $\Delta _{C_1},\Delta _{C_2}$
extend smoothly through $\vec{0}$ by the previous paragraph.
This contradicts the usual maximum principle for $H$-surfaces, as
both $\Delta _{C_1}$, $\Delta _{C_2}$ have the same orientation at $\vec{0}$
since the orientation of the multigraph component $C$ of
type 4 induces the orientation of both $\Delta _{C_1},\Delta _{C_2}$.
\end{proof}

By Observation {\bf (O2)} and  Assertions~\ref{ass5.4} and~\ref{ass5.5}, every leaf of ${\cal L}$
which limits to $\vec{0}$ is a proper annulus which extends smoothly across $\vec{0}$
(hence by the maximum principle there are at most two of them, with common tangent
plane $\Pi $ at $\vec{0}$ and oppositely pointing mean curvature vectors) and there exists at least
one such proper annulus. Therefore, property {\bf (P)} can now be improved to the following property:
\par
\vspace{.2cm}
\noindent
{\bf (P)'} {\it Under rescaling by every sequence
$\{\lambda_n\}_n\subset (0,\infty )$ with $\lambda_n\to \infty$, a
subsequence of the weak $\frac{H}{\l _n}$-laminations
${\cal L}_n =\lambda _n{\cal L}\subset \lambda_n B_N (p,r)$
converges in $\R^3-\{ \vec{0}\} $ to a lamination ${\cal L}'$ of
$\rth$ by planes parallel to $\Pi $.
}
\par
\vspace{.2cm}
\noindent

Let $F$ be one of the at most two proper annular leaves in ${\cal L}$ limiting
to $\vec{0}$.
Let $\overline{F}$ be the
extended $H$-disk obtained after attaching the origin to $F$. Consider an intrinsic geodesic disk
$D_{\overline{F}}(\vec{0},\de)$ in $\overline{F}$ centered at $\vec{0}$ with radius $\de $,
and let $\eta $ be the unit normal vector field
 to $D_{\overline{F}}(\vec{0},\de )$. Pick coordinates $q=(x,y)$ in $D_{\overline{F}}(\vec{0},\de)$
and let $t\in [-\tau ,\tau ]\mapsto
\g _q(t)=q+t\eta (q)$ be the straight line in $\R^3$
passing through $q$ with velocity vector $\eta (q)$
(here $\tau >0$ is small and independent of
$q\in D_{\overline{F}}(\vec{0},\de )$ so that the straight lines
$\g _q$ do not intersect each other).
Then for some $\tau >0$ small,
$(x,y,t)$ produces ``cylindrical'' normal coordinates in a neighborhood $V$ of $\vec{0}$ in $\R^3$,
and we can consider the natural projection
\[
\Phi \colon V\to D_{\overline{F}}(\vec{0},\de ),\ \Phi (x,y,t)=(x,y).
\]
Since by {\bf (P)'} every blow-up limit of ${\cal L}$ from $\vec{0}$
is a
lamination of $\R^3$ by planes parallel to $\Pi $,
we conclude that for $\de $ and $\tau $ sufficiently small,
the angle of the intersection of any leaf component $L_V$ of ${\cal L}\cap V$
with any straight line $\g _q$ as above can be made arbitrarily close to $\frac{\pi }{2}$. Taking
$\de $ much smaller than $\tau $, a monodromy argument implies that
any leaf component $L_V$ of
${\cal L}\cap V$ which contains a point at distance at most $\frac{\de }{2}$ from $\vec{0}$ is a graph over
$D_{\overline{F}}(\vec{0},\de )$; in other words, $\Phi $ restricts to $L_V$ as a diffeomorphism
onto $D_{\overline{F}}(\vec{0},\de )$.
\begin{assertion}
\label{ass5.6}
There exists a uniform bound around $\vec{0}$ for the function $|\sigma _{\cal F}|$ defined in (\ref{eq:sigma})
(note that this property will complete the proof of Proposition~\ref{propos5.2}).
\end{assertion}
\begin{proof}
Reasoning by contradiction, assume that there exists a sequence of points $p_n$ in leaves
$L_n$ of ${\cal L}$ converging to $\vec{0}$, such that $|\sigma _{L_n}|(p_n)$ diverges.
Without loss of generality, we can assume that $p_n\in V$ and $p_n$ is a point where
the following function attains its maximum:
\[
f_n\colon L_n\cap V\to [0,\infty ), \quad f_n(a)
=|\sigma _{L_n}|(a)\, d_{\overline{F}}(\Phi (a),\partial D_{\overline{F}}(\vec{0},\de )),
\]
where $d_{\overline{F}}$ denotes the intrinsic distance in
$\overline{F}$ to the boundary $\partial D_{\overline{F}}(\vec{0},\de )$.
 Now expand the above coordinates $(x,y,z)$ centered at $\vec{0}$ with ratio
$|\sigma _{L_n}|(p_n)\to \infty $. Under this expansion, $V$ converges
to $\R^3$ with its usual flat metric and the straight lines $\g _q$
converge to parallel lines.
The graphical property that $\Phi $ restricts to any leaf component $L_V$ of ${\cal L}\cap V$
$\frac{\de }{2}$-close to $\vec{0}$ as a diffeomorphism
onto $D_{\overline{F}}(\vec{0},\de )$ gives that after passing to a subsequence, the
$H$-graphs $L_n\cap V$ converge after expansion of coordinates
to a minimal surface in $\R^3$ which is an entire graph. By the Bernstein Theorem, such
a limit surface is a flat plane. This contradicts that the  ratio of the homothetic expansion
coincides with the norm of the second fundamental form of $L_n\cap V$ at $p_n$ for all $n$.
This contradiction finishes the proof of Assertion~\ref{ass5.6},
and completes the proof of Proposition~\ref{propos5.2}.
\end{proof}

\begin{proposition}
\label{propos5.7}
Theorem~\ref{tt2} holds in the general case for the ambient manifold $N$.
\end{proposition}
\begin{proof}
In the manifold setting for $N$, under rescaled exponential
coordinates from $p$ we have the same description as in the proof of Proposition~\ref{propos5.2},
and the arguments in that proof adapt with straightforward modifications; also see Cases IV and V of the proof
of Theorem~1.1 in~\cite{mpr10}. This completes the proof
of Theorem~\ref{tt2}.
\end{proof}

We next extend Corollary~7.1 in~\cite{mpr10} to the case of a weak $H$-lamination in a Riemannian
three-manifold. We remark that the statements in items~4, 5
of the corollary below do not have corresponding statements in Corollary~7.1 in~\cite{mpr10}.
Regarding item~5 of Corollary~\ref{corrs} and using its notation, we make the following definition.
\begin{definition}
\label{def5.8}
  {\rm
  The absolute mean curvature function of a weak CMC foliation ${\cal F}$ of $N-W$ is
  the function $|H_{\cal F}|\colon N-W\to [0,\infty )$ defined by
\[
|H_{\cal F}|(p)=\sup \{ |H_L|\ | \ L \mbox{ is a leaf of ${\cal F}$ passing through $p$}\} .
\]
Note that as in the case of $|\sigma _{\cal F}|$ given by (\ref{eq:sigma}), the function $|H_{\cal F}|$ is not necessarily continuous.
}
\end{definition}
The hypothesis of boundedness of $|H_{\cal F}|$  in item 5 of Corollary~\ref{corrs}
is essential: take
$N=\R^3$, $W=\{ \vec{0}\} $ and ${\cal F}$ the foliation of $\R^3-\{ \vec{0}\} $ by concentric spheres.

\begin{corollary}
\label{corrs}
Let $H\in \R $. Suppose that $N$ is a Riemannian
three-manifold, not necessarily complete. If $W\subset N$ is a closed countable subset and
${\cal L}$ is a weak $H$-lamination of $N-W$ such that for every $p\in W$ there exists
positive constants $\ve ,C$ (possibly depending on $p$) satisfying the following curvature estimate:
\begin{equation}
\label{eq:ce}
 |\sigma _{\cal L}|(q)\, d_N(q,W)\leq C\quad \mbox{ for all }q\in B_N(p,\ve )-W,
\end{equation}
then ${\cal L}$ extends across $W$ to a weak $H$-lamination of~$N$.
In particular:
\begin{enumerate}
\item The closure of any collection of the stable leaves of a weak $H$-lamination of $N-W$
extends across $W$ to a weak $H$-lamination of $N$ consisting of stable $H$-surfaces.

\item  The closure 
in $N$ of any collection of limit leaves of  a weak $H$-lamination ${\cal L}$  of $N-W$
is a  weak $H$-lamination of $N$, all whose leaves are stable $H$-surfaces.

\item If ${\cal F}$ is a weak $H$-foliation of $N-W$, then ${\cal F}$
extends across $W$ to a weak $H$-foliation of~$N$.

\item If ${\cal F}$ is a weak CMC foliation of $N-W$ and $H\in \R $, then the closure in
$N-W$ of any collection ${\cal F}(H)$ of leaves of ${\cal F}$ with constant curvature $H$
extends across $W$ to a weak $H$-lamination of $N$.

\item If ${\cal F}$ is a weak CMC foliation of $N-W$ with bounded absolute
mean curvature function, then ${\cal F}$ extends across $W$ to a weak CMC foliation of $N$.
\end{enumerate}
\end{corollary}
{\it Proof.}
Let ${\cal L}$ be a weak $H$-lamination of $N-W$ satisfying the curvature estimate (\ref{eq:ce}),
where $W$ is closed and countable. Since the extension of ${\cal L}$ across $W$
is a local question, it suffices to extend ${\cal W}$ in small, open extrinsic balls in $N$. Since
$W$ is countable, we can take these balls so that each of their boundaries
are disjoint from $W$, and their closures in $N$ are compact. It follows
that for every such ball $B_N$, the set $W\cap B_N$ is a complete
countable metric space. By Baire's Theorem, the set $W_0$ of isolated points of the locally compact
metric space $W\cap B_N$ is dense in $W\cap B_N$.

\begin{assertion}
\label{ass5.9new}
  In the above situation, ${\cal L}\cap B_N$ extends across $W_0$
to a weak $H$-lamination of $B_N-(W-W_0)$.
\end{assertion}
\begin{proof}
Consider an isolated point $p\in W\cap B_N$. By hypothesis, there exist $\ve ,C>0$
such that the inequality~(\ref{eq:ce}) holds. Taking $\ve >0$ smaller if necessary, we can
assume that the closed ball $\overline{B}_N(p,\ve )$ is compact and contained in $B_N$, its
boundary $S^2_N(p,\ve )$ is  disjoint from $W$ and that $B_{N}(p,\ve) \cap W=\{ p\} $.
By Theorem~\ref{tt2}, the induced local weak $H$-lamination ${\cal L}\cap  [B_{N}(p,\ve)-\{ p\} ]$
extends across $p$ to a weak $H$-lamination of $B_{N}(p,\ve)$. This proves the assertion.
\end{proof}

Consider the collection $\mathfrak{U}$ of open subsets
$U$ of $B_N$ such that $B_N-W\subset U$ and there exists a weak $H$-lamination
${\cal L}_U$ of $U$ whose restriction to $B_N-W$ coincides with ${\cal L}|_{B_N-W}$.
By Assertion~\ref{ass5.9new}, $B_N-(W-W_0)\in \mathfrak{U}$.
We claim that $\bigcup _{U\in \mathfrak{U}}U\in \mathfrak{U}$.
Note that if $U\in \mathfrak{U}$, then the related weak $H$-lamination ${\cal L}_U$
is unique (since leaves of ${\cal L}_U$ are analytic surfaces that coincide with
the leaves of ${\cal L}|_{B_N-W}$).  Given $U_{\a },U_{\be }\in \mathfrak{U}$
and given a point $x\in U_{\a }\cap U_{\be }$, the related weak $H$-laminations
${\cal L}_{\a },{\cal L}_{\be }$ that extend
${\cal L}|_{B_N-W}$ to $U_{\a }, U_{\be }$ satisfy
${\cal L}_{\a }|_{U_{\a }\cap U_{\be }}=
{\cal L}_{\be }|_{U_{\a }\cap U_{\be }}$ by the above uniqueness property.
Therefore, $\bigcup _{\a \in \Lambda }U_{\a }\in \mathfrak{U}$ and our claim is proved.

We want to prove that if $V:=\bigcup _{U\in \mathfrak{U}}U\in \mathfrak{U}$, then $V=B_N$, which will finish the proof of the first statement of the corollary.
Arguing by contradiction, suppose $B_N-V\neq \mbox{\O }$. Since $B_N-V \subset W\cap B_N$ is a non-empty closed subset of $W\cap B_N$, then $B_N-V$ is a complete countable
metric space and so, Baire's theorem again insures that the set $I$
of its isolated points is dense
in $B_N-V$. By Assertion~\ref{ass5.9new}, the $H$-lamination ${\cal L}_V$ obtained by extension of ${\cal L}$ to $V$
extends through every isolated point of $B_N-V$; hence $V\cup I\in \mathfrak{U}$. By definition of $V$, this implies that $V\cup I\subset V$, hence $I\subset V$. As $I\subset B_N-V$, then $I=\mbox{\O }$ which contradicts that $I$ is dense in $B_N-V$.
Now the proof of the first statement of Corollary~\ref{corrs} is complete.

Item 1 of the corollary follows from the already proven first statement and from
curvature estimates for stable $H$-surfaces (Schoen~\cite{sc3}, Ros~\cite{ros9}, see also
Theorem~2.15 in~\cite{mpr19}). By Theorem 4.3 of~\cite{mpr19} (see also Theorem~1 in~\cite{mpr18}),
limit leaves of a weak $H$-lamination are stable (if $H\neq 0$ they are two-sided;
in the minimal case, the two-sided cover of every limit leaf is stable). As the
collection of limit leaves of a weak $H$-lamination is closed, then item~2 of the corollary
follows from item~1. Item~3 is a direct consequence of item~2, as every leaf of a weak
$H$-foliation is a limit leaf.

To prove item 4, let ${\cal F}$ be a weak CMC foliation of $N-W$,
where $W$ is closed and countable, and let $H\in \R $. Reasoning as in the case of a
weak $H$-lamination, we can reduce the proof of the extendability of
any collection ${\cal F}(H)$ of leaves of ${\cal F}$ with constant mean curvature $H$
to the case in which ${\cal F}$ is  a weak CMC foliation
of a small open extrinsic ball $B_N$ with compact closure in $N$, such that
$W\cap S^2_N=\mbox{\O }$. Also the above argument based on
Baire's Theorem 
allows one to reduce the proof to the case that $B_N=B_N(p,\ve )$
where $p\in W$ is an isolated point of $W$ and $B_{N}(p,\ve) -\{p\}\subset
\Int(N)-W$. To prove that ${\cal F}(H)$ extends across $p$, it suffices to show
that for some small $\ve>0$, the induced local weak $H$-lamination ${\cal F}_1(H)={\cal F}(H)\cap
[B_{N}(p,\ve) -\{p\}]$
extends across $p$ to a weak $H$-lamination of $B_{N}(p,\ve)$.

Consider the weak CMC foliation ${\cal F}_1={\cal F}\cap [B_N(p,\ve )-\{ p\} ]$.
By the universal curvature estimate in Theorem~\ref{thm5.7}
applied to each of the compact three-manifolds with boundary $N(k)=
\overline{B_N}(p, \ve )-B_N(p,\frac{\ve }{k})$, $k\in \N$,
there exists a constant $A>0$ independent of
$k$ such that for each $k$,
we have
\begin{equation}
\label{eq:handbook}
|\sigma _{{\cal F}_1}|(q)\leq \frac{A}{\min\{{\rm dist}_N(q,\partial N(k)),
\frac{\pi }{\sqrt{\Lambda }}\}},
\qquad \mbox{ for all }q\in \mbox{Int}[N(k)],
\end{equation}
where $\Lambda \geq 0$ is an upper bound of the sectional curvature of
$N$ in $\overline{B}_N(p, \ve )$ and $|\sigma _{{\cal F}_1}|(q)$ is defined in
(\ref{eq:sigma}). Taking $\ve $ smaller if necessary (this does not change the
constant $\Lambda $), we can assume that $d_N(q,\partial N(k))\leq \frac{\pi }{\sqrt{\Lambda }}$
for all $k\in \N$. Thus, given $k\geq 3$ and $q\in N(k)\cap B_N(p,\frac{\ve }{2})$, we have
\begin{eqnarray*}
|\sigma _{{\cal F}_1}|(q)\, d_N(q,p)&=&
|\sigma _{{\cal F}_1}|(q)\, \min\{{\rm dist}_N(q,\partial N(k)),
{\textstyle \frac{\pi }{\sqrt{\Lambda }}}\} \, \frac{d_N(q,p)}{d_N(q,\partial N(k))}
\\
&\stackrel{(\ref{eq:handbook})}{\leq }&
A\, \frac{d_N(q,p)}{d_N(q,\partial N(k))}
\\
&=&A\, \frac{d_N(q,p)}{d_N(q,S^2_N(p,\ve /k))}\stackrel{(k\to \infty )}{\longrightarrow }A.
\end{eqnarray*}
Hence, the weak $H$-lamination ${\cal F}_1(H)$ satisfies the curvature estimate in
the hypothesis of Theorem~\ref{tt2}, and thus, ${\cal F}_1(H)$ extends across
$p$ as desired. This proves item~4 of the corollary.

Finally we prove item~5. Let ${\cal F}$ be a weak CMC foliation of $N-W$
with bounded absolute mean curvature function, where $W$ is closed and countable.
Similar arguments as in the previous cases show that we can reduce the proof of item~5
to the proof of the extendability of a weak CMC foliation ${\cal F}$ of a small open
extrinsic ball $B_N(p,\ve )$ with compact closure in $N$, where $p\in W$ is isolated in $W$
 and $\overline{B}_N(p,\ve )-\{ p\} \subset \mbox{Int}(N)-W$.
 To prove that ${\cal F}$ extends across $p$, it suffices to prove
that for some smaller $\ve>0$, the induced local weak CMC foliation ${\cal F}_1={\cal F}\cap
[B_N(p,\ve )-\{ p\} ]$ extends across $p$ to a weak CMC foliation of $B_{N}(p,\ve)$.
Since by hypothesis the absolute mean curvature function of the leaves of ${\cal F}_1$ is bounded,
we can choose $\ve >0$ sufficiently small so that for all $\de \in (0,\ve ]$, the
absolute mean curvature function of the (smooth) distance sphere $S^2_N(p,\de )$
is strictly greater than the maximum value of the absolute mean curvature of the leaves of
${\cal F}_1$. As an application of the mean curvature comparison
principle, we conclude that the closure of every leaf of ${\cal F}_1$ intersects $S_N^2(p,\ve )$
(see Claim~A in the proof of Lemma~\ref{ass:flat} for a similar argument).

Take a sequence $\{ p_n\} _n\subset B_N(p,\ve )-\{ p\} $ converging to $p$ as $n\to \infty $.
As ${\cal F}_1$ is a weak CMC foliation of $B_N(p,\ve )-\{ p\} $, for each $n\in \N$ there
exists at least one leaf $L_n$ of ${\cal F}_1$ with $p_n\in L_n$. Let $H_n$ be the
(constant) mean curvature of $L_n$ and let $H=\lim \sup H_n$. After replacing by a subsequence, we
may assume that $H=\lim _nH_n$. Let ${\cal F}_1(H)$ be the weak $H$-lamination of $B_N(p,\ve )-\{ p\} $
consisting of all leaves of ${\cal F}_1$ whose mean curvature is $H$.

We claim that $p$ lies in the closure of ${\cal F}_1(H)$ in $B_N(p,\ve )$.
To see this it suffices to show that given $k\in \N $, some leaf of ${\cal F}_1(H)$ intersects
$S_N^2(p,\frac{\ve }{k})$. Fix $k\in \N$. As $\{ p_n\} _n\to p$, then for $n$ sufficiently
large $p_n\in B_N(p,\frac{\ve }{k})$. As the closure of $L_n$ intersects $S^2_N(p,\ve )$
and $L_n$ is connected, then $L_n$ also intersects $S^2_N(p,\frac{\ve }{k})$. For each
$n\in \N$ large, pick a point $x_n\in L_n\cap S^2_N(p,\frac{\ve}{k})$. Since
$S^2_N(p,\frac{\ve }{k})$ is compact, after extracting a subsequence, the
$x_n$ converge  as $n\to \infty$ to a point $x\in S^2_N(p,\frac{\ve }{k})$.
As the mean curvatures of the $L_n$ converge to $H$, then
there passes a leaf $\wh{L}$ of ${\cal F}_1(H)$ through $x$, and our claim is proved.

\begin{assertion}
\label{ass5.11}
The weak CMC foliation ${\cal F}_1$ extends across $p$ to a
weak CMC foliation of $B_N(p,\ve )$ (and thus, the proof of item~5
of Corollary~\ref{corrs} is complete).
\end{assertion}
\begin{proof}
By the last claim and the already proven item~4 of this corollary,
${\cal F}_1(H)$ extends across $p$ to a weak $H$-lamination of $B_N(p,\ve)$.
Let $\overline{L}$ be the leaf of the extended weak $H$-lamination ${\cal F}_1(H)\cup \{ p\} $
passing through $p$ (thus, $L=\overline{L}-\{ p\} $ is a leaf of ${\cal F}_1$).
After possibly choosing a smaller $\ve$, we
may assume that $\overline{L}$ is a smooth embedded disk in
$\overline{B}_N(p,\ve)$ with compact boundary in $S^2_N(p,\ve )$. Using again the
curvature estimates (\ref{eq:handbook}) we get that for any
sequence of positive numbers $\lambda_n \to \infty$, a subsequence
of the punctured balls $\lambda_n[B_N(p,\ve)-\{ p\} ]$ converges as $n\to \infty $
to $\rth -\{ \vec{0}\} $ with its usual metric, and the
weak CMC foliations $\l _n{\cal F}_1$ converge to a limit weak CMC
foliation ${\cal F}_{\infty}$ of $\R^3-\{ \vec{0}\} $,
which is in fact a minimal foliation since the mean curvature of the leaves of ${\cal F}_1$ is
bounded. Note that one of the leaves of ${\cal F}_{\infty}$ is the
punctured plane $\Pi$ passing through $\vec{0}$, corresponding to the
blow-up of the tangent plane to the disk $\overline{L}$ at $p$. By item~3 of this corollary,
${\cal F}_{\infty}$ extends across the origin to a minimal foliation of
$\R^3$; since every leaf of this extended minimal foliation is a complete
stable minimal surface in $\R^3$, then every such leaf is a plane, which
must be parallel to $\Pi$. In particular, the
limit foliation ${\cal F}_{\infty}$ is independent of the sequence
$\lambda_n \to \infty$. In this situation, it follows that for $\ve$ sufficiently
small, the leaves in ${\cal F}_1$ can be uniformly locally expressed as non-negative
or non-positive normal graphs with bounded gradient over their
projections to $\overline{L}$. In particular, there is a weak CMC
foliation structure on ${\cal F}_1\cup \{ p\} $, and Assertion~\ref{ass5.11} is proved.
\end{proof}

\section{Proof of Theorem~\ref{thmspheresintrod}.}
\label{sectspheres}


We start by proving the $\rth$-version of Theorem~\ref{thmspheresintrod} in the more
general setting of weak CMC foliations.

\begin{theorem}
\label{thmspheres}
Suppose that ${ \cal F}$ is
a weak  CMC foliation of $\rth$ with a  closed
countable set $ \cal S$ of singularities (these are the points
where the weak CMC structure of ${\cal F}$ cannot be extended).  Then,
each leaf of ${\cal F}$ is contained in either a plane or a round sphere,
and $0\leq |{\cal S}|\leq 2$. Furthermore
if $\cal S$ is empty, then
$\cal F$ is a foliation by planes.
\end{theorem}
{\it Proof.}
Note that if all leaves of ${\cal F}$ are minimal,
then ${\cal F}$ is a minimal foliation of $\R^3-S$, hence by item~3 of
Corollary~\ref{corrs},
${\cal F}$ extends to a minimal foliation of $\R^3$, which must then consist entirely of
parallel planes and the theorem holds in this case.
Therefore, in the sequel we may assume that ${\cal F}$ contains a leaf which is not minimal.

\begin{assertion}
\label{ass6.2}
Every non-minimal leaf of ${\cal F}$ is proper in $\R^3-{\cal S}$.
Furthermore, if ${\cal S}$ is bounded, then every non-minimal leaf of ${\cal F}$
is contained in a ball.
\end{assertion}
\begin{proof}
Consider a leaf $L$ of ${\cal F}$, with mean curvature $H$.
By item~4 of Corollary~\ref{corrs}, the collection ${\cal F}(H)$ of $H$-leaves in ${\cal F}$
extends across ${\cal S}$ to weak $H$-lamination of $\R^3$ and so, the
closure of $L$ in $\R^3$ is a weak $H$-lamination of $\R^3$.
If $H\neq 0$ and $L$ is not proper in $\R^3-{\cal S}$, then $\overline{L}$ contains
a limit leaf $L_1$, which is complete since $\overline{L}$ is a weak $H$-lamination
of $\R^3$. By Theorem~4.3 of~\cite{mpr19} (see also Theorem~1 in~\cite{mpr18}) applied to the
weak $H$-lamination $\overline{L}$,
$L_1$ is stable ($L_1$ is two-sided since its mean curvature is non-zero).
This contradicts that there are no stable complete $H$-surfaces in $\R^3$ for any
$H\neq 0$. This proves the first sentence in the assertion.

Next suppose ${\cal S}$ is bounded and take a non-minimal leaf $L\in {\cal F}$. If
there exists an extrinsically divergent sequence
of points $p_n\in L$, then the extrinsic distance $d_n$ from $p_n$ to ${\cal S}$
tends to infinity as $n\to \infty $ and hence, Corollary~\ref{cor:scale} applied to the
weak CMC foliation ${\cal F}\cap \B (p_n,d_n)$ of $\B (p_n,d_n)$ implies that
the second fundamental form of $L$ at $p_n$
decays to zero in norm, as $|\sigma _L|\leq |\sigma _{\cal F}|$.
In particular, the trace of the second fundamental form of $L$ must be zero since $L$ has constant
mean curvature, which gives a contradiction. Therefore, every non-minimal leaf $L$ of ${\cal F}$
lies in some ball of $\R^3$.
\end{proof}

\begin{assertion}
\label{ass6.3}
If a leaf $L$ of ${\cal F}$ is contained in a ball of $\R^3$, then its closure
$\overline{L}$ is a round sphere.
\end{assertion}
\begin{proof}
As in the proof of the previous assertion, the closure of $L$ in $\R^3$
has the structure of a weak $H$-lamination of $\R^3$ by item~4 of Corollary~\ref{corrs}.
Since there are no
bounded minimal laminations in $\R^3$ by the maximum principle, then $H\neq 0$.
By Assertion~\ref{ass6.2}, $L$ is proper in $\R^3-{\cal S}$, and thus,
$\overline{L}$ consists of a
single compact immersed surface which does not intersect itself transversely,
and whenever $\overline{L}$ intersects itself, it locally consists of two
disks with opposite mean curvature vectors. Hence, $\overline{L}$ is Alexandrov-embedded.
In this situation, Alexandrov~\cite{aa1} proved that
$\overline{L}$ is a round sphere.
\end{proof}

\begin{assertion}
\label{ass6.4}
If ${\cal S}$ is bounded, then Theorem~\ref{thmspheres} hold.
\end{assertion}
\begin{proof}
As ${\cal S}$ is bounded, then Assertions~\ref{ass6.2} and~\ref{ass6.3} give that
the closure of every non-minimal leaf of ${\cal F}$ is a round sphere.
Consider the collection ${\cal A}$ of all
spherical leaves of $\cal F$ union with ${\cal S}$.
Then, the restriction of ${\cal F}$ to the complement of the closure of ${\cal A}$
is a minimal foliation ${\cal F}_1$ of the open set $\R^3-(\overline{\cal A}\cup {\cal S})$. Applying item~3
of Corollary~\ref{corrs} to ${\cal F}_1$, $N=\R^3-{\cal A}$ and $W={\cal S}\cap N$, we conclude that
${\cal F}_1$ extends across ${\cal S}\cap N$ to a minimal foliation of $\R^3-{\cal A}$.
By item~4 of Corollary~\ref{corrs}, ${\cal F}_1$ extend across ${\cal S}$
to a minimal lamination of $\R^3$, and thus the extended leaves of ${\cal F}_1$ are complete.
As ${\cal F}_1$ consists of stable leaves by Theorem~1 in~\cite{mpr18}, then
item~1 of Corollary~\ref{corrs} insures that the extended leaves of ${\cal F}$ across ${\cal S}$
are complete stable minimal surfaces in $\R^3$, hence planes. As the weak CMC foliation ${\cal F}$ is now entirely formed by
punctured spheres and planes, then it is clear that ${\cal S}$ consists of one of two points.
This completes the proof of Assertion~\ref{ass6.4}.
\end{proof}

To prove Theorem~\ref{thmspheres} in the general case of a closed countable
set ${\cal S}\subset \R^3$,
we next analyze the structure of ${\cal F}$ in a neighborhood of an
isolated point $p\in {\cal S}$ (recall that the set of isolated points
in ${\cal S}$ is dense in ${\cal S}$ by  Baire's Theorem). Since $p$
is isolated in ${\cal S}$, we can choose a sphere $\esf^2(p,r)$ such
that $\overline{\B }(p,r)\cap {\cal S}=\{ p\} $. As ${\cal S}$ is
closed, then $\esf^2(p,r)$ is at positive distance from ${\cal S}$.
Since $|\sigma _{\cal F}|$ is locally bounded in $\R^3- {\cal S}$
(by definition of weak CMC lamination), an elementary
compactness argument shows that 
there is a uniform upper bound for the restriction to $\esf ^2(p,r)$
of the norms of the second
fundamental forms of all leaves in ${\cal F}$ which intersect $\esf
^2(p,r)$; in particular the absolute mean curvature of every such
leaf satisfies $|H|\leq C_1$ for some $C_1>0$. By item~5 of Corollary~\ref{corrs},
the mean curvature of the leaves of ${\cal F}$ is
unbounded in every neighborhood of $p$, since $p\in {\cal S}$.
Therefore, there exist leaves of ${\cal F}$ which intersect $\B
(p,r)$ and whose mean curvatures satisfy $|H|>C_1$. Every
such leaf $L$ is entirely contained in $\B (p,r)$ and thus,
Assertion~\ref{ass6.3} implies that the closure $\overline{L}$ of $L$ in $\R^3$ is a round sphere.
Note that either $p\in \overline{L}$ or $p$ lies in the open ball $\B _L$ enclosed by
$\overline{L}$ (otherwise a monodromy argument shows that ${\cal F}\cap \B _L$
is a ``product'' foliation by spheres, which produces a singularity $q\in \B _L$;
this contradicts that ${\cal S}\cap \B (p,r)=\{ p\} $).

The above arguments show that for every
isolated point $p$ of ${\cal S}$, one of the two following possibilities holds:
\begin{description}
\item[(A)]
There exists an open neighborhood $V_p$ of $p$ in $\R^3$ such that ${\cal F}$
restricts to $V_p-\{ p\} $ as a weak CMC foliation by round spheres and
$V_p\cap {\cal S}=\{ p\} $, see Figure~\ref{figure5new} left.
\item [(B)]
There exists an open ball $\B (q,R)\subset
\R^3$ such that $p\in \esf^2(q,R)$ and the weak CMC foliation
${\cal F}$ restricts to $\overline{\B } (q,R)-\{ p\} $ as a
union of round spheres punctured at $p$,
all tangent at $p$. In this case, we call
$V_p=\B (q,R)$, see Figure~\ref{figure5new} right.
\end{description}
\begin{figure}
\begin{center}
\includegraphics[height=5cm]{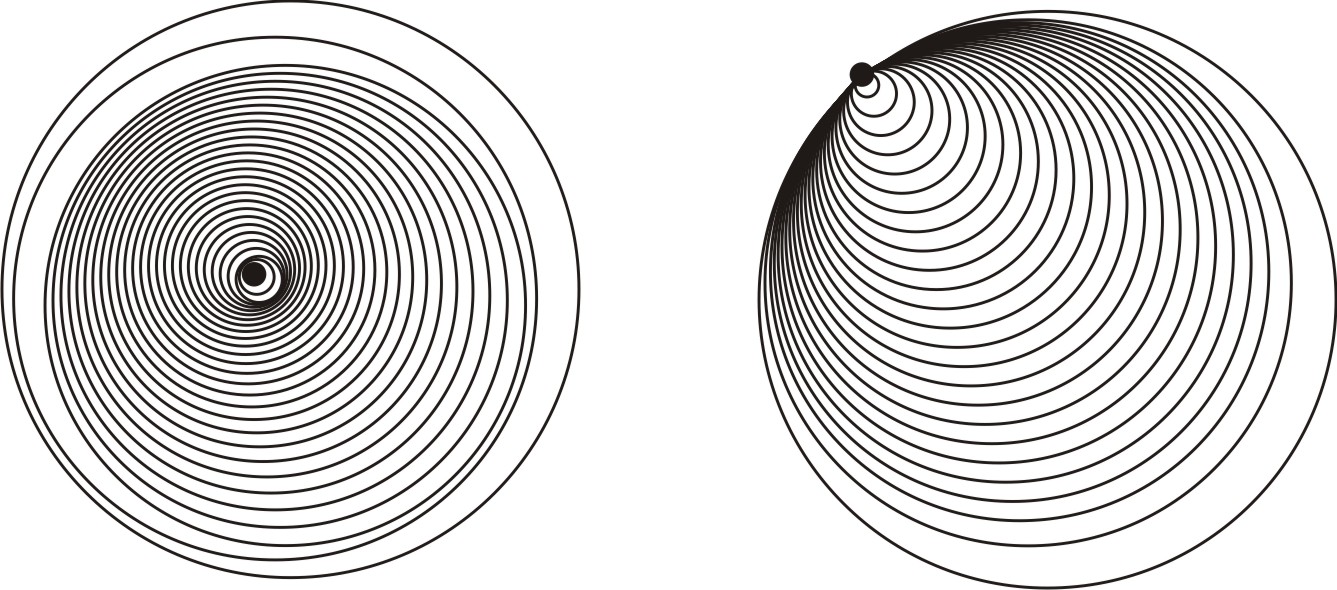}
\caption{Left: Case {\bf (A)} of the proof of
Theorem~\ref{thmspheres}. Right: Case {\bf (B)}. In both cases,
the dot represents an isolated singular point $p\in {\cal S}$.}
\label{figure5new}
\end{center}
\end{figure}
If possibility {\bf (A)} occurs for an isolated point $p$ of ${\cal S}$,
we define $U_p$ to be the maximal such open set $V_p$, with the ordering
given by the inclusion. Note that in this case, we have two mutually
exclusive possibilities:
\begin{description}
\item[(A1)] The boundary $\partial U_p$ of $U_p$ is empty; in this case, $U_p=\R^3$
and Theorem~\ref{thmspheres} is proved with ${\cal S}=\{ p\} $.
\item[(A2)] The boundary $\partial U_p$ is non-empty; in this case $\partial U_p$ is either a
round sphere (and $U_p$ is an open ball of $\R^3$ containing $p$),
or $\partial U_p$ is a plane (and $U_p$ is an open half-space containing $p$).
In both of these cases, $U_p$ only intersects ${\cal S}$ at the point $p$.
\end{description}
If possibility {\bf (B)} holds for an isolated point $p$ of ${\cal S}$,
we define $U_p$ to be the union of the maximal open 1-parameter family of spheres in ${\cal F}$,
possibly punctured at $p$, that contains the open ball $V_p$ described
in possibility {\bf (B)}, together with the point $p$ if this
union contains a spherical leaf of ${\cal F}$
that does not pass through $p$. As in case {\bf (A)}, we have two mutually
exclusive possibilities:
\begin{description}
\item[(B1)] The boundary $\partial U_p$ of $U_p$ is empty; in this case, $U_p=\R^3$
and Theorem~\ref{thmspheres} is again proved with ${\cal S}=\{ p\} $. Therefore, in the sequel we will
assume that for each isolated point $p\in {\cal S}$, we have $\partial U_p\neq \mbox{\O }$.
\item[(B2)] The boundary $\partial U_p$ is non-empty; in this case $\partial U_p$ is either a
round sphere (and in $U_p$ is an open ball of $\R^3$ with $p$ in its closure),
or $\partial U_p$ is a plane (and $U_p$ is an open halfspace with $p$ in its
closure). In both of these cases, $U_p$ only intersects ${\cal S}$ in
at most the point $p$.
\end{description}

Note that the case of two simultaneous
such maximal open sets $U_p\neq U_p'$ can occur in case  {\bf (B)}; for instance when ${\cal F}$
is the foliation of $\R^3-\{ \vec{0}\} $ given by the $(x_1,x_2)$-plane together
with all spheres passing through $p=\vec{0}$ and tangent to the $(x_1,x_2)$-plane
(in this case $U_p=\{ x_3>0\} $ and $U_p'=\{ x_3<0\} $). In the case that we have
 two possibilities for choosing
$U_p$, we will simply arbitrarily choose one such $U_p$ in our discussions below.

We next collect some elementary properties of these open sets $U_p$, which easily follow from
the fact that ${\cal F}$ is a foliation outside ${\cal S}$ and from the
description in possibilities {\bf (A)},
{\bf (B)} above.

\begin{description}
  \item[(P1)] If $p,q$ are distinct isolated points of ${\cal S}$, then $U_p\cap U_q=\mbox{\O }$.
  \item[(P2)] If $\{ p_n\} _n$ is a converging sequence of distinct isolated points of ${\cal S}$, then
  for $n$ sufficiently large, $U_{p_n}$ is an open  ball and the radii of the $U_{p_n}$ converge to zero.
\end{description}

Note that by maximality and a standard
monodromy argument, if $\partial U_p$ is a sphere then $\partial U_p
\cap {\cal S}\neq \mbox{\O }$.

\begin{assertion}
\label{ass10.5}
Given an isolated point $p\in {\cal S}$, suppose that
$U_p$ is an open ball. Then,  $\partial U_p\cap {\cal S}$
contains at least one point which is not isolated in ${\cal S}$.
\end{assertion}
\begin{proof}
Recall that $\partial U_p
\cap {\cal S}\neq \mbox{\O }$. Arguing by contradiction,
suppose that $\partial U_p\cap {\cal S}$ consists only of isolated points of
${\cal S}$. In particular, $\partial U_p \cap {\cal S}$ is finite, say
$\partial U_p \cap {\cal S}=\{ p_1,\ldots ,p_k\} $. Note that $p$ lies in
$\partial U_p \cap {\cal S}$ if and only if possibility {\bf (B)} above holds for $p$.
Then,  the above arguments show that around every point
$p_1\in\partial U_p \cap {\cal S}$, necessarily Case {\bf (B)}
occurs (exchanging $p$ by $p_1$), and that the following additional
property holds:
\begin{description}
\item[(P3)] If $p_1\neq p$, then
the related maximal open set $U_{p_1}$ is disjoint from $U_p$ and every leaf
in the restriction of ${\cal F}$ to $\overline{U_{p_1}}$ is a punctured
sphere or punctured plane whose closure only intersects $\overline{U_p}$ at $p_1$.
\end{description}
Analogously, if $p\in \partial U_p\cap {\cal S}$ is an isolated point where we have
two possibilities $U_p,U_p'$ for choosing $U_p$, then the same property {\bf (P3)}
holds for $p_1=p$ and $U_{p_1}=U_p'$, see Figure~\ref{newfigure}.
\begin{figure}
\begin{center}
\includegraphics[width=12cm]{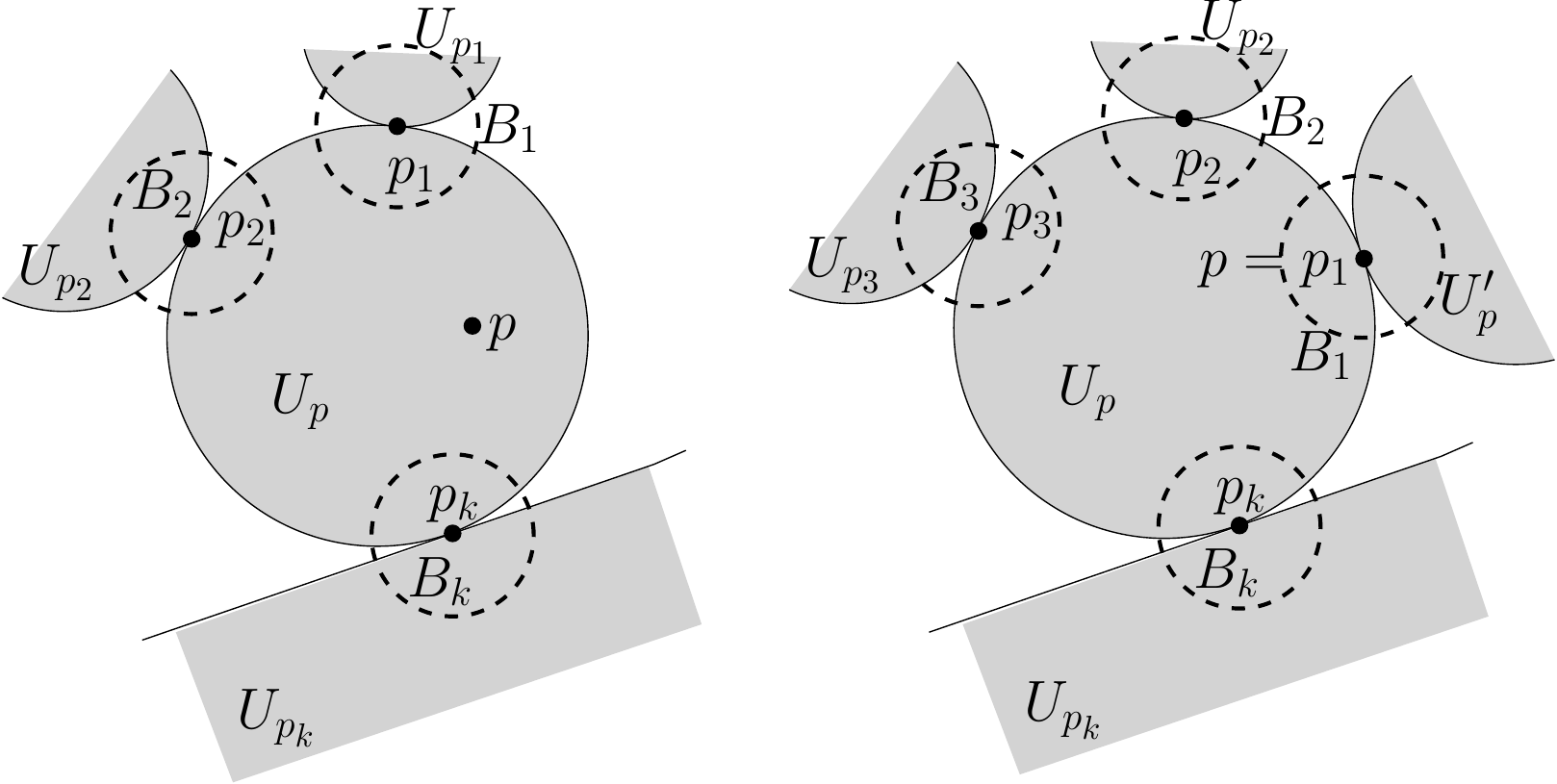}
\caption{On the left, possibility {\bf (A)} occurs for $p$. On the right,
possibility {\bf (B)} holds for $p$ and there are two possible choices
$U_p$,$U'_p$ for $U_p$. In both figures, we have represented one of
the points in $\partial U_p\cap {\cal S}$ (namely, $p_k$) so that the related
set $U_{p_k}$ is a halfspace.}
\label{newfigure}
\end{center}
\end{figure}

Let $B_1,\ldots ,B_k\subset \R^3$ be pairwise disjoint, small open balls centered at
the points $p_1,\ldots ,p_k$. As
the $p_i$ are isolated in ${\cal S}$, we can assume that ${\cal S}\cap \left(
\bigcup _{i=1}^kB_i\right) =\partial U_p\cap {\cal S}$. We denote by $D_i=
B_i\cap \partial U_p$, which is a spherical disk. Next we will prove the following property.
\begin{description}
\item[(P4)] Given $i=1,\ldots ,k$, if the radius of the ball $B_i$ is small enough, then
the intersection of ${\cal F}$ with the region $W_i=
B_i-[U_{p_i}\cup U_p]$ consists of a collection of annuli,
each of which can be expressed as a normal graph over its projection to $D_i-\{p_i\}$.
\end{description}
To see why {\bf (P4)} holds, we first prove that under
 blow-up from $p_i$, the induced foliation
${\cal F}\cap W_i$ converges smoothly to the punctured tangent plane to $\partial U_p$ at $p_i$.
Note that the region $W_i$ converges after such a blow-up to
the punctured tangent plane to $\partial U_p$ at $p_i$. Since we have a scale-invariant
uniform bound on the second fundamental of $({\cal F}\cap B_i)-\{ p_i\} $
(given by Theorem~\ref{thm5.7}, see the proof of Assertion~\ref{ass6.2} for a similar argument),
then the leaves of the induced foliation
${\cal F}\cap W_i$ are locally graphical over small geodesic disks of $\partial U_p-\{ p_i\} $
contained in $D_i$. It remains to check that these local graphs, when glued together inside a leaf of
${\cal F}\cap W_i$, do not define multi-valued graphs over the punctured spherical disk
$D_i-\{ p_i\} $. Arguing by contradiction, suppose that there exists a leaf $L_{W_i}$ of
${\cal F}\cap W_i$ which can be expressed by a multi-valued graph (not univalent) over
$D_i-\{ p_i\} $. Consider a small compact solid cylinder $C_i$ whose axis is the normal line to
$\partial U_p$ passing through $p_i$, such that both $\partial U_p$ and $\partial U_{p_i}$
intersect $C_i$ in compact closed disks bounded by distinct parallel circles.
Then, the intersection of $L_{W_i}$ with
$\partial C$ contains a spiraling curve $\G $ (with infinite length in both directions)
which is trapped between the circles
$\partial U_p\cap C$, $\partial U_{p_i}\cap C$. The curve $\G $ limits to two
disjoint closed curves
$\wh{\G }_1,\wh{\G }_2\subset \partial C$, which are topologically parallel to
$\partial U_p\cap C$, $\partial U_{p_i}\cap C$. Let $\wh{L}_1,\wh{L}_2$ be the leaves
of ${\cal F}$ that contain $\wh{\G }_1$, $\wh{\G }_2$, respectively (note that
a priori, $\wh{L}_1$ could coincide with $\wh{L}_2$). As $\wh{\G }_1,\wh{\G }_2$
consist of limit points of $\G $, then $\wh{L}_1$, $\wh{L}_2$ are limit leaves of
the sublamination ${\cal F}(H)$ of ${\cal F}$ consisting of the leaves of ${\cal F}$
with the same mean curvature $H$ as $L_{W_i}$. By item~4 of Corollary~\ref{corrs},
both $\wh{L}_1$, $\wh{L}_2$ extend smoothly across $p_i$. Since the $H$-surfaces
$\wh{L}_1\cup \{ p_i\} $, $\wh{L}_2\cup \{ p_i\} $ have the same mean curvature
vectors at $p_i$ (their orientations are induced by the one of the multi-valued graph
inside $L_{W_i}$), we contradict the maximum principle for $H$-surfaces.
Now property {\bf (P4)} is proved.

Since {\bf (P4)} holds, it follows that $\partial U_p-{\cal S}$,
considered to be a leaf of the weak CMC
foliation ${\cal F}$ of $\R^3-{\cal S}$, has trivial holonomy on
the exterior side of $\partial U_p$.
This implies that the leaves of ${\cal F}$ nearby $\partial U_p-{\cal S}$
and outside $U_p$ are topologically
punctured spheres which are graphs over $\partial U_p-{\cal S}$.
Therefore, these graphs extend
smoothly to embedded topological spheres with constant mean curvature;
hence the extended graphs are themselves round spheres.
This contradicts the maximality of $U_p$. This contradiction finishes
the proof of Assertion~\ref{ass10.5}.
\end{proof}

\begin{assertion}
  \label{ass6.6}
Theorem~\ref{thmspheres} holds.
\end{assertion}
\begin{proof}
Consider the set ${\cal S}_0$ of those points of ${\cal S}$ which are isolated;
recall that ${\cal S}_0$ is an open dense subset of ${\cal S}$ by Baire's Theorem.
If for every $p\in {\cal S}_0$ the related maximal open set $U_p$ given
just before Assertion~\ref{ass10.5} is a halfspace, then clearly
${\cal S}$ only consists of one or two singularities, and by Assertion~\ref{ass6.4},
the theorem holds in this case.
So it suffices to show that for every $p\in {\cal S}_0$, $U_p$ cannot
be an open ball. Arguing by contradiction, suppose that there exists
$p\in {\cal S}_0$ such that $U_p$ is an open ball. Since ${\cal S}-
{\cal S}_0\neq \mbox{\O }$ by Assertion~\ref{ass10.5} and ${\cal S}-{\cal S}_0$
is closed in ${\cal S}$, then the set
${\cal S}_1$ of isolated points in ${\cal S}-{\cal S}_0$ is non-empty
(in fact, ${\cal S}_1$ is dense in ${\cal S}-{\cal S}_0$ by
Baire's Theorem applied to the complete metric space ${\cal
S}-{\cal S}_0$ together with the collection of open dense subsets of
${\cal S}-{\cal S}_0$ given by $A_n=({\cal S}-{\cal S}_0)-\{
p_1,\ldots ,p_n\} $, where $\{ p_n\ | \ n\in \N \} $ is an
enumeration of the countable set ${\cal S}-({\cal S}_0\cup {\cal
S}_1)$). Pick a point $q\in {\cal S}_1$, which must be a limit of
a sequence of points $p_n\in {\cal S}_0$.
Since the $p_n$ converge to $q$, property {\bf (P2)} insures that for $n$ large,
$U_{p_n}$ is an open ball and the radii of $U_{p_n}$ converge to zero as $n\to \infty $.
Since each $\partial U_{p_n}\cap
{\cal S}$ contains at least one point in ${\cal S}-{\cal S}_0$ (by
Assertion~\ref{ass10.5}), then we contradict that $q$ is isolated in
${\cal S}-{\cal S}_0$. This contradiction proves that $U_p$ cannot
be an open ball, and therefore finishes the proof of Assertion~\ref{ass6.6}.
\end{proof}

We next prove the $\esf^3$-version of Theorem~\ref{thmspheres}.
\begin{theorem}
\label{thmspheres2}
Let ${ \cal F}$ be
a weak  CMC foliation of $\esf^3$ with a  closed
countable set $ \cal S$ of singularities (as in Theorem~\ref{thmspheres}, these
singularities are the points
where the weak CMC foliation structure of ${\cal F}$ cannot be extended).  Then,
each leaf of ${\cal F}$ is contained in a round sphere in $\esf^3$ and the number of
 singularities is $|{\cal S}|=1$ or $2$.
\end{theorem}
\begin{proof}
This theorem can be proven with minor modifications
of the proof in the $\R^3$ case. One starts by proving that
if ${\cal F}$ is a weak CMC foliation of $\esf^3-{\cal S}$
with ${\cal S}$ closed, then ${\cal S}$ cannot be empty. This holds
since otherwise, as $\esf^3$ is compact and the absolute mean curvature
function of ${\cal F}$ is locally bounded, there exists a leaf in ${\cal F}$
of maximal absolute mean curvature, which
by Proposition~5.4 in~\cite{mpr19} must be stable. This contradicts the non-existence of
complete, stable surfaces with constant mean curvature in $\esf^3$.

By item~4 of Corollary~\ref{corrs}, for any $H\in \R$, the
weak $H$-lamination ${\cal F}(H)$ of ${\cal F}$ consisting of all leaves of constant mean curvature
$H$ extends smoothly across the  singular set ${\cal S}$ of ${\cal F}$ to a collection
${\cal F}(H)'=\{ \overline{L}\ | \ L\in {\cal F}(H)\} $ of compact immersed $H$-surfaces.
The non-existence of limit leaves of ${\cal F}(H)'$ implies that this collection of surfaces
in finite. 

Next we show that $\overline{L}$ is Alexandrov-embedded
for every leaf $L$ of ${\cal F}(H)$. Observe that $\overline{L}$ cannot have transversal intersections,
as ${\cal F}$ is a weak CMC foliation. If $\overline{L}$ is not embedded, then $H\neq 0$
and there exists $p\in \overline{L}$ such that locally around $p$, $\overline{L}$ consists of two
disks that lie at one side of each other with opposite mean curvature vectors. As $\overline{L}$
is compact, there exists some small $\ve >0$ such that
\[
\overline{L}(\ve )=\{ \exp _p(tN_p)\ | \ p\in \overline{L}, t\in [0,\ve )\}
\]
is an embedded $\ve $-neighborhood
on the mean convex side of $\overline{L}$, where $N$ stands for the unit normal vector field
to $\overline{L}$ for which $H$ is the mean curvature of $\overline{L}$.
In particular, the parallel surface $\overline{L}_{\ve /2}$ at distance $\frac{\ve}{2}$ from
$\overline{L}$ inside $\overline{L}(\ve)$ is embedded.
As every compact embedded surface separates $\esf^3$, then $\overline{L}_{\ve /2}$ divides $\esf^3$ into
two open domains, one of which, called $\Omega $, contains $\partial \overline{L}(\ve )-\overline{L}$.
The union of the closure of $\Omega $ in $\esf^3$ with the $\frac{\ve }{2}$-neighborhood
$\overline{L}(\frac{\ve }{2})\subset \overline{L}(\ve )$ of $\overline{L}$ can be viewed
as the image of a submersion of a
3-manifold with boundary into $\esf^3$, with its boundary image being $\overline{L}$. This
proves that $\overline{L}$ is Alexandrov-embedded.

Let $p\in {\cal S}$ be an isolated point in ${\cal S}$, which exists by Baire's Theorem.
As $p$ is isolated in ${\cal S}$, we can choose a geodesic sphere $S^2(p,r)$ in $\esf^3$
of small radius $r\in (0,\frac{\pi }{2})$
centered at $p$ such that the geodesic ball $B(p,r)$ enclosed by $S^2(p,r)$ satisfies
$B(p,r)\cap {\cal S} = \{ p\} $. As ${\cal S}$ is closed, then $S^2(p,r)$ is at positive distance
from ${\cal S}$. Since $|\sigma _{\cal F}|$ is locally bounded in $\esf^3 -{\cal S}$, an
elementary compactness argument shows that there is a uniform upper bound for the restriction
to $S^2(p,r)$ of the norms of the second fundamental forms of all leaves in ${\cal F}$ which intersect
$S^2(p,r)$; in particular the absolute mean curvature of every such leaf satisfies $|H|\leq C$
for some $C>0$. By item 5 of Corollary~\ref{corrs}, the mean curvature of the leaves of ${\cal F}$
is unbounded in every neighborhood of $p$, since $p\in {\cal S}$. Therefore, there exist leaves of
${\cal F}$ which intersect $B(p,r)$ and whose mean curvatures satisfy $|H| > C$. Every such leaf
$L$ is then entirely contained in $B(p,r)$. As $r<\frac{\pi }{2}$, then $\overline{L}$ is contained in a hemisphere,
and by a standard application of the Alexandrov moving plane technique, we conclude that
$\overline{L}$ is a sphere. As we did in the case of $\R^3$, one can consider the
maximal collection $U_p$ of round spheres in ${\cal F}$ around every isolated
point of ${\cal S}$. In this setting, Assertion~\ref{ass10.5} remains valid with
only straightforward modifications in its proof, as well as the Baire's Theorem
argument in the proof of Assertion~\ref{ass6.6}. We leave the remaining details for the reader.
\end{proof}
\begin{remark} {\em
Recall that by Theorem~\ref{thmspheres}, if a weak CMC foliation $\cF$ of $\rth$ has a countable
set of singularities $\cS$ (as a weak foliation), then $\cS$ consists of  at most two points.
If $|\cS|=0$, then $\cF$ is a foliations by parallel planes.
If $|\cS|=1$, then up to a translation we can assume $\cS=\{\vec{0}\}$ and we have two
cases:
\begin{enumerate}
\item Some leaf of $\cF$ has the origin in its closure. In this case, up to a rotation of
$\rth$ fixing $\vec{0}$, there are exactly two examples $\cF_1,\cF_2$, where
$\cF_1$ is the set of horizontal planes $\{x_3=t \mid t\leq 0\}$ together with the spheres
$\esf^2(p_t,t)$ with $p_t=(0,0,t)$ for all $t>0$,
and $\cF_2$ is the set of spheres tangent to the $(x_1,x_2)$-plane at $\vec{0}$
together with the  $(x_1,x_2)$-plane. In particular, in this Case 1 the weak foliation $\cF$ of
$\rth-\{\vec{0}\}$ is actually a foliation of  $\rth-\{\vec{0}\}$.
\item No leaf of $\cF$ has the origin in its closure, as in the particular foliation $\cF_0$
of $\R^3-\{ \vec{0}\} $ consisting of the set of spheres centered at the origin.
In this Case 2, the weak foliation $\cF$ can be isotoped to $\cF_0$
fixing the origin along the isotopy, and the spherical leaves can intersect themselves as
leaves of a weak foliation.
\end{enumerate}
In Theorem~\ref{corol6.6} below we will describe the structure of a singular weak CMC
foliation $\cF$  of a three-manifold in a neighborhood of an isolated singularity $p$
as being modeled on a weak singular CMC foliation of $\rth-\{\vec{0}\}$ with $|\cS|=1$;
in fact we will show that the weak foliation $\cF$ is closely approximated by exactly one
of the examples in the two cases above.
%
}
\end{remark}

\section{Structure of singular CMC foliations in a neighborhood of an isolated singularity.} \label{sec:localstructure}
Theorem~\ref{thmspheres} implies that the two possibilities represented in Figure~\ref{figure5new}
give canonical models for every weak CMC foliation
of $\R^3-\{ \vec{0}\} $ as a 1-parameter collection of spheres and planes with
one singularity occurring at the origin.
In fact, this description is a good model for the local structure
of any weak CMC foliation in a Riemannian three-manifold around an isolated
singularity, as the following theorem demonstrates.
\begin{theorem}
\label{corol6.6}
Let $B_N(p,r)$ be a metric ball in a Riemannian three-manifold and
${\cal F}$ be a weak CMC foliation of $B_N(p,r)-\{ p\} $.
Then:
\begin{enumerate}
  \item ${\cal F}$ extends across $p$ to a weak
  CMC foliation of $B_N(p,r)$ if and only if the
 absolute mean curvature function of ${\cal F}$ (see Definition~\ref{def5.8}) is bounded.
  \item  Suppose that the absolute mean curvature function of ${\cal F}$ is unbounded in
  $B_N(p,\frac{r}{2}) $.
  Then, there exists $r_0\in (0, r)$
such that:
  \begin{description}
  \item[{\it (2A)}] For every sequence $\l _n>0$
  with $\l _n\to \infty $, the blow-up weak CMC foliations $\l _n[{\cal F}\cap B_N(p,r_0)]$
  of $\l _nB_N(p,r_0)-\{ p\} $
  converge (after extracting a subsequence) to a non-flat weak CMC foliation of $\R^3-\{ \vec{0}\} $.
  \item[{\it (2B)}] There exists  $H_0>0$ such that the closure in $B_N(p,r_0)$
  of every  leaf of ${\cal F}\cap B_N(p,r_0)$ with absolute mean curvature $H> H_0$ is an embedded
  $H$-sphere that bounds a subball of $B_N(p,r)$ that contains $p$ in its closure.
  \end{description}
\end{enumerate}
\end{theorem}
{\it Proof.} Item~1 of this theorem follows directly from item~5 of Corollary~\ref{corrs}.

In the sequel we will assume that the absolute mean curvature function of ${\cal F}$
is unbounded in $B_N(p,\frac{r}{2}) $, and so,
for every $r_0\in (0,r)$ the absolute mean curvature function of ${\cal F}\cap B_N(p,r_0)$ is
also unbounded. Choose $r_0\in (0,r)$ sufficiently small so that $r_0$ is less than
the injectivity radius function of $N$ at $p$ and so that the geodesic spheres in
$B_N(p,r_0)$ centered at $p$ have positive (not necessarily constant) mean curvature
with respect to the inward pointing normal vector.
By Theorem~\ref{thm5.7}, ${\cal F}\cap B_N(p,r_0)$ satisfies the curvature estimate
$|\sigma_{\cal F}|(q)\, d_N(q,p)\leq C$ for some constant
$C>0$ independent of $q\in B_N(p,r_0)-\{ p\} $, where $|\sigma _{\cal F}|$ is given by
(\ref{eq:sigma}) (see the proof of item~4 of Corollary~\ref{corrs} for a similar argument).
It follows from this curvature estimate that
blow-up rescalings of ${\cal F}\cap B_N(p,r_0)$ of the form $\lambda_n [{\cal F}\cap B_N(p,r_0)]\subset
\lambda_n [B_N(p, r_0)-\{p\}]$ for $\lambda_n\to \infty$, have subsequences that
converge to a weak CMC foliation ${\cal F}_{\infty}$ of
$\rth-\{\vec{0}\}$, which gives item~(2A) of the corollary, except for the condition that the
limit CMC foliation ${\cal F}_\infty$ is non-flat,
whose proof we postpone for the moment.


\begin{assertion}
\label{ass6.8new}
For $r_0>0$ small enough, the ball $B_N(p,r_0)$ contains no complete $H$-surfaces
whose two-sided cover is stable, for any value of $H\in \R $.
\end{assertion}
\begin{proof}
Fix $r_0>0$ small, satisfying the properties previous to the assertion. By the main theorem
 in~\cite{rst1}, there exists a uniform bound for the second fundamental form of any complete
$H$-surface in $B_N(p,r_1)$ whose two-sided cover is stable, which is independent of the
value of~$H$. In particular, $|H|$ is also bounded from above by a universal constant $H'>0$.
Now choose $r_1\in (0,r_0)$ small enough so that the distance spheres
$S^2_N(p,\de )$, $\de \in (0,r_1]$, all have mean curvature strictly greater
than $H'$, and suppose that $B(p,r_1)$ contains a complete $H$-surface $\Sigma $ whose
two-sided cover is stable. Then, there exists a point $q$ in the closure of $\Sigma $ which is at
maximal distance from $p$, and a sequence of small intrinsic open disks $D_n\subset \Sigma $
that converge to an open $H$-disk $D_{\infty }$ passing through $q$. As $D_{\infty }$ lies in
the closure of $\Sigma $, then $D_{\infty }$ lies at one side of the distance sphere
$S^2_N(p,d_N(p,q))$ at $q$. The mean comparison principle applied to $D_{\infty }$ and
$S^2_N(p,d_N(p,q))$ gives a contradiction as the mean curvature of $S^2_N(p,d_N(p,q))$ is strictly
less than $H$. This proves the assertion, after relabeling $r_1$ by $r_0$.
\end{proof}

\begin{assertion}
\label{ass6.8}
Take $r_0>0$ as in Assertion~\ref{ass6.8new}.
Let $H_0=H_0(r_0)$  be the supremum of the mean curvatures of the leaves of ${\cal F}$ that
intersect $S^2_N(p,r_0)$. Let
\[
{\cal A}(r_0)=\{ L \mbox{ leaf of }{\cal F}\cap
 B_N(p,r_0) \ : \ |H_L|>H_0\} ,
\]
where $|H_L|$ is the absolute mean curvature of $L$. Then, the closure
$\overline{L}$  in $B_N(p,r_0)$ of any $L\in {\cal A}(r_0)$ is a compact immersed surface.
\end{assertion}
\begin{proof}
 Take $L\in {\cal A}(r_0)$. As  $L$ is a leaf of a weak CMC foliation, item~4 of
Corollary~\ref{corrs} implies that the weak $H_L$-lamination $\overline{L}-\{ p\} $ of
$B_N(p,r_0)-\{ p \} $ extends across $p$ to a weak $H_L$-lamination of $B_N(p,r_0)$, all whose
leaves are complete (note that the extrinsic distance from the lamination $\overline{L}$ to
$S^2_N(p,r_0)$ is positive). If $L$ were not proper, then the lamination
$\overline{L}$ would contain a complete leaf whose two-sided cover is stable, which
contradicts  Assertion~\ref{ass6.8new}. Hence, $L$ is proper and so, $\overline{L}$ is a connected, compact immersed $H_L$-surface in $B_N(p,r_0)$.
\end{proof}

\begin{assertion}
\label{ass6.9}
With the notation of Assertion~\ref{ass6.8} and after possibly choosing a
smaller value of $r_0$, $\overline{L}$   is a compact embedded surface in  $B_N(p,r_0)$
for every $L\in {\cal A}(r_0)$.
\end{assertion}
\begin{proof}
First observe that if a codimension-one CMC foliation $\wh{{\cal F}}$ of a Riemannian manifold~$N$
is transversely orientable\footnote{A codimension-one foliation $\wh{{\cal F}}$ of a manifold $N$ is called
 {\it transversely orientable} if $N$ admits a continuous, nowhere zero vector field
whose integral curves are transverse to the leaves of $\wh{{\cal F}}$.}, then the
leaves of ${\cal F}$ are embedded; this follows from the fact that locally around a
point $q$ of self-intersection of a leaf $L$ of ${\cal F}$ with itself, $L$ consists of two small
embedded disks tangent at $q$ with non-zero opposite mean curvature vectors at $q$.
This contradicts the fact that the mean curvature vector of $L$ equals
$H_L\cdot (N_{\cal F})|_L$(up to sign), where $N_{\cal F}$ is a continuous unitary
vector field on $N$ orthogonal to the leaves of ${\cal F}$.

We now prove the assertion. Since $B_N(p,r_0)-\{ p\} $ is simply connected, then
${\cal F}\cap [B_N(p,r_0)-\{ p\} ]$ is transversely oriented. Thus, Assertion~\ref{ass6.8} and
the observation in the last paragraph imply that if $L\in {\cal A}(r_0)$, then $\overline{L}$ is a
compact immersed surface and $L$ is embedded. So it remains to show that for $r_0$
sufficiently small, every leaf $L$ in ${\cal A}(r_0)$ with $p\in \overline{L}$ satisfies that
$\overline{L}$ is embedded at $p$. Arguing by contradiction, suppose that for each $n\in \N$,
there exists a leaf $L_n$ in ${\cal A}(\frac{r_1}{n})$ such that the compact immersed surface
$\overline{L}_n$ is embedded outside $p$ and $p$ is a point of self-intersection of
$\overline{L}_n$. In particular, $\overline{L_n}$ has non-zero mean curvature and the mean
curvature vectors of $\overline{L_n}$ point in opposite directions at the two points of the
abstract surface $\overline{L_n}$ occurring at the self-intersection point $p$; this means that
we can consider two local intrinsic disks $D_n,\wh{D}_n\in \overline{L_n}$, that intersect only
at $p$, and which have opposite mean curvatures with respect to a fixed orientation of their
common tangent planes at $p$.

In particular, for some $n_0$ large, the surface $\overline{L_1}$ intersected with
$B_N(p,\frac{r_0}{n_0})$ consists of two disks, each of whose boundary curves lies in
$S^2_N(p,\frac{r_0}{n_0})$; let $D$ be one of these two disks. Choose $n_1>n_0$ such that the
absolute mean curvature of  $\overline{L_{n_1}}\subset B_N(p,\frac{r_0}{n_1})\subset
B_N(p,\frac{r_0}{n_0})$ is greater than the absolute mean curvature of $D$. Then,
$D_{n_1},\wh{D}_{n_1}\in \overline{L_{n_1}}$ are tangent to $D$ at $p$ and lie on the
same side of $D$, since $\overline{L_{n_1}}$ is connected and lies in one of the two
closed complements of $D$ in $B_N(p,\frac{r_0}{n_0})$. But  then at the point $p$, the disk $D$ lies
on the mean convex side of one of the disks  $D_{n_1},\wh{D}_{n_1}$, which contradicts
the mean curvature comparison principle since the absolute mean curvature of $D$ is less than
that of $\overline{L_{n_1}}$. This contradiction finishes the proof of the assertion.
\end{proof}

\begin{assertion}
\label{ass6.9bis}
Given $L\in {\cal A}(r_0)$, let $\Delta(\overline{L})$ be the compact subdomain of
$B_N(p,r_0)$ bounded by the compact embedded surface $\overline{L}$.  Then,
$\Delta(\overline{L})$  contains the point $p$.
\end{assertion}
\begin{proof}
Arguing by contradiction, suppose that $p\not \in \Delta(\overline{L})$. Note that the mean
convex side of the compact embedded surface $\overline{L}$ is
$\Delta (\overline{L})-\overline{L}$ (this follows from considering the innermost distance
sphere $S^2_N(p,\de)$ such that $\overline{L}\subset B_N(p,\de )$ and comparing the
mean curvature vectors of $\overline{L}$ and of $S^2_N(p,\de )$). As ${\cal F}$ does not have
any singularities in $\Delta (\overline{L})$, then there exists a leaf $L_1$ of ${\cal F}$ which
maximizes the mean curvature function of ${\cal F}$ restricted to $\Delta (\overline{L})$.
If $L_1=L$, then Proposition~5.4 in~\cite{mpr19} implies that $L$ is stable, which contradicts
Assertion~\ref{ass6.8new}. If $L_1\neq L$, then we can apply the same
Proposition~5.4 in~\cite{mpr19} on the mean convex side of $L_1$, which is strictly
contained in $\Delta (\overline{L})$ to contradict  Assertion~\ref{ass6.8new} again.
This proves the assertion.
%
%
\end{proof}

\begin{assertion}
\label{ass6.10}
For $r_0$ sufficiently small, the closure $\overline{L}$ of every $L\in {\cal A}(r_0)$
is topologically a sphere. 
\end{assertion}
\begin{proof}
Consider the distance function in $N$ to $p$, $d_p\colon B_N(p,r_0)\to [0,r_0)$.
By Assertion~\ref{ass6.9bis}, one of the following two exclusive possibilities holds
when $r_0$ is sufficiently small:
\begin{enumerate}[(A)]
\item $p\in \Int( \Delta(\overline{L}))$ for every $L\in {\cal A}(r_0)$ (or equivalently,
$L=\overline{L}$ for every $L\in {\cal A}(r_0)$ by Assertion~\ref{ass6.9bis}).
\item $p\in \overline{L}$ for every $L\in {\cal A}(r_0)$.
\end{enumerate}
Suppose first that we are in case (A) and pick $L\in {\cal A}(r_0)$. The restriction
$f=(d_p)|_L$ is strictly positive as $L$ is compact and $p\notin L$. For simplicity, we will
assume that $f$ is a Morse function, although this is not strictly necessary (also, one could
 perturb slightly $f$ to a Morse function so that the argument that follows remains valid).
Let $\wh{f}\colon \Delta (L)\to \R$ be the restriction of $d_p$ to  $\Delta (L)$. If $L$ is not
topologically a sphere, then $\Delta (L)$ is not a closed topological ball and thus, there exists
a critical point $q$ of $f$ in $L$ such that $\wh{f}^{-1}(0,f(q)-\ve ]$ is not homeomorphic
to $\wh{f}^{-1}(0,f(q)+\ve ]$ for all sufficiently small $\ve >0$. At such a critical point $q$,
$L$ is tangent to the distance sphere $S_N^2(p,f(q))$ at $q$ and the mean curvature
vector of $S_N^2(p,f(q))$ points outward from the mean convex side of $L$ at $q$ (as $f$ is
assumed to be a Morse function, then $f$ has index 1 at $q$). By contradiction, suppose that
the assertion fails in this case (A). Thus there exists a sequence $r_n\searrow 0$, leaves
$L_n\in {\cal A}(r_n)$, none of which is a sphere, and critical points $q_n\in L_n$ of
$f_n=\wh{f}|_{L_n}$ such that the mean curvature vector of $L_n$ at $q$ points outward
$B_N(p,d_p(q_n))$ at $q_n$. After rescaling by $\l _n=\frac{1}{d_p(p_n)}$ and extracting
a subsequence, the weak CMC foliations $\l _n{\cal F}_n$ converge as $n\to \infty $ to a
weak CMC foliation ${\cal F}_{\infty }$ of $\R^3-\{ \vec{0}\} $, whose leaves have closures
being  spheres or planes by Theorem~\ref{thmspheresintrod}. Furthermore,
${\cal F}_{\infty }$ contains a leaf $L_{\infty }$ that passes through the limit point
$q_{\infty }\in \esf^2(1)$ as $n\to \infty $ of the points $\l _nq_n$. In particular,
the closure $\overline{L_{\infty }}$ of $L_{\infty }$ is a plane or a sphere passing through
$q_{\infty }$. As $L_n$ is tangent to $S_N^2(p_n,d_p(q_n))$ at $q_n$, then $L_{\infty }$ is
tangent to $\esf^2(|q_{\infty }|)$ at $q_{\infty }$. Therefore, we have three possibilities for
$L_{\infty }$:
\begin{enumerate}
\item $L_{\infty }$ is a plane passing through $q_{\infty }$, orthogonal to this position vector.
\item $L_{\infty }$ is a sphere passing  through $q_{\infty }$ with mean curvature vector
pointing in the direction of $q_{\infty }$.
\item $\overline{L_{\infty }}$ is a sphere passing  through $q_{\infty }$ with mean curvature
vector pointing in the direction of $-q_{\infty }$.
\end{enumerate}
Also, observe that as the mean curvature vectors of $L_n$, $S_N^2(p,d_p(q_n))$ at their
common point $q_n$ point in opposite directions, then Case 3 above cannot occur. Cases 1, 2
cannot occur either, as then the distance function in $\R^3$ to $q_{\infty }$ has a
non-degenerate (global) minimum at $q_{\infty }$, which contradicts that $f_n$ has index 1 at
$q_n$ for all $n$. Hence the assertion holds  for $r_0$ sufficiently small in case (A).

In case (B) the argument is similar to the previous one, and we leave the details
to the reader. This finishes the  proof of the assertion.
\end{proof}

We next finish the proof of Theorem~\ref{corol6.6}. Assertions~\ref{ass6.9bis} and
 \ref{ass6.10} imply that item~(2B) of the corollary holds.
It only remains to show that given $\l _n>0$ with $\l _n\to \infty $, the weak CMC foliation
${\cal F}_{\infty}$ of $\rth-\{\vec{0}\}$ obtained as a limit of a subsequence of
$\l _n[ {\cal F}\cap B_N(p,r_0)]$ in the first paragraph of the proof of this corollary, is non-flat.
Fix $r_0>0$ small enough so that Assertion~\ref{ass6.10} holds.

\begin{assertion}
\label{ass7.15}
For $n$ large enough, there exists $L\in {\cal A}(r_0)$ such that $\Delta (L)\subset
\overline{B}_N(p,\frac{1}{\l _n})$ and $L\cap S^2_N(p,\frac{1}{\l _n})\neq \mbox{\rm \O }$.
\end{assertion}
\begin{proof}
Assume that $n$ is chosen sufficiently large so that $\frac{1}{\l _n}<r_0$ and so that
any leaf of ${\cal A}(r_0)$ that is contained in
$\overline{B}_N(p,\frac{1}{\l _n})\subset \overline{B}_N(p,r_0)$ has mean curvature greater than
$2H_0$ (recall that $H_0$ was defined in the statement of Assertion~\ref{ass6.8});
this choice of $n$ is possible since every leaf in ${\cal A}(r_0)$
contained in $\overline{B}_N(p,\frac{1}{\l _n})$ is compact embedded sphere
and for $n$ large,  the mean curvature such a leaf  is at least $\frac{1}{2\l _n}$.

Consider the non-empty open set
\[
A_-=\bigcup _{L\in {\cal A}_-}\mbox{Int}[\Delta (L)], \quad
\mbox{where}\quad {\cal A}_-=\{ L\in {\cal A}(r_0)\ | \ \Delta (L)\subset \overline
{B}_N(p,{\textstyle \frac{1}{\l _n}})\} .
\]
We claim that there exists $L_1\in {\cal A}(r_0)$ such that $L_1\subset \partial (A_-)$ and $\Delta
(L_1)\subset \overline{B}_N(p,\frac{1}{\l _n})$. To see this, take a point $q\in \partial (A_-)-\{ p\} $ and
a sequence of points $q_k\in A_-$ converging to $q$ as $k\to \infty $. Then, $q_k$ lies in the
interior of $\Delta (L_k)$ for some $L_k\in {\cal A}(r_0)$ with $\Delta (L_k)\subset
\overline{B}_N(p,\frac{1}{\l _n})$. As the norm of the second fundamental form of leaves
of ${\cal F}$ is uniformly bounded locally around $q$, then all $L_k$ can be locally expressed
around $q$ as graphs
of uniform size over their tangent planes at $q_k$. By the Arzel\`{a}-Ascoli theorem,
a subsequence of these graphs converges to a graph of constant mean curvature at least $2H_0$, and a
 monodromy argument shows that this limit graph is contained in a spherical leaf
$L_1\in {\cal A}(r_0)$ with $\Delta (L_1)\subset \overline{B}_N(p,\frac{1}{\l _n})$, so our claim is proved.

We will show that $L_1$ intersects $S^2_N(p,\frac{1}{\l _n})$ (so the assertion will be proved by
taking $L=L_1$). Arguing by contradiction, suppose $L_1\cap S^2_N(p,\frac{1}{\l _n})=\mbox{\O }$.
Now consider the family
\[
{\cal A}_+=\{ L\in {\cal A}(r_0)\ | \ L-\overline{B}_N(p,{\textstyle \frac{1}{\l _n}})\neq \mbox{\O } \mbox{ and }  \Int(\Delta (L)) \cap\Delta (L_1)\neq \O\} ,
\]
which we claim is non-empty. The arguments that demonstrate that this set is non-empty are more or less
the same as the ones given in the previous paragraph; one shows that for some point $q\in L_1-\{p\}$ there is a sequence of points
$q_k\in \overline{B}_N(p,{\textstyle \frac{1}{\l _n}})-\Delta(L_1) $ converging to $q$ and contained in respective leaves
$L_k$ with mean curvatures converging to the mean curvature of $L_1$, and so the leaves $L_k$  lie in ${\cal A}_+$ for $k$ large.

Since the family of balls $\{ \Delta (L)\ | \ L \in {\cal A}(r_0) \mbox{ and }
\Int(\Delta (L)) \cap\Delta (L_1)\neq \O\} $
is totally ordered under inclusion, it follows that for every $s\in \N$ and
$L^1,\ldots ,L^s\in {\cal A}_+$, we have $\bigcap _{i=1}^s[\Delta (L^i)\cap S^2_N(p,
\frac{1}{\l _n})]\neq \mbox{\O }$. Since $S^2_N(p,\frac{1}{\l _n})$ is compact, we conclude that there
exists a point $Q$ such that
\begin{equation}
\label{ass7.15A}
Q\in \bigcap _{L\in {\cal A}_+}[\Delta (L)\cap S^2_N(p,{\textstyle \frac{1}{\l _n}})].
\end{equation}
As ${\cal F}$ is a weak CMC foliation, there exists a leaf $L_Q\in {\cal F}$ passing through
$Q$, although $L_Q$ might not be unique. Let ${\cal F}_Q$ be the collection of leaves
of ${\cal F}$ passing through $Q$. Since $S^2_N(p,\frac{1}{\l _n})\subset
\mbox{Int}[\Delta (L')]$ and $Q\in S^2_N(p,\frac{1}{\l _n})$, then every such leaf
$L_Q\in {\cal F}_Q$ lies in ${\cal A}(r_0)$ (and thus, the closure of $L_Q$ is a sphere).
Also, every two leaves $L_Q,L'_Q\in {\cal F}_Q$ intersect tangentially, with one at one side
of the other; therefore either $\Delta (L_Q)\subset \Delta (L'_Q)$ or vice versa.  As there
exists a uniform local bound around $Q$ for the norms of the second fundamental forms
of leaves in ${\cal F}_Q$, we conclude that the union of the leaves in ${\cal F}_Q$
is a compact set of $N$. Thus, we can choose $\wt{Q}\in B_N(p,\frac{1}{\l _n})$ with the following
properties:
\begin{enumerate}[(P1)]
\item $d_N(Q,\wt{Q})<\frac{\de}{2}$, where $\de =d(Q,L_1)=d(Q,A_-)>0$.
\item $\wt{Q}\in \mbox{Int}[\Delta (L_Q)]$, for all $L_Q\in {\cal F}_Q$.
\end{enumerate}
As before, there exists $\wt{L}\in {\cal A}(r_0)$ such that $\wt{Q}\in \wt{L}$. Note that $Q\notin \wt{L}$ (otherwise $\wt{L}\in {\cal F}_Q$, which contradicts (P2) above). Therefore,
(P2) implies that $Q\notin \Delta (\wt{L})$. In turn, this implies by (\ref{ass7.15A}) that
$\wt{L}\notin {\cal A}_+$. As $\wt{L}\in {\cal A}(r_0)$, then by definition of ${\cal A}_+$
we have $\wt{L}\subset \overline{B}_N(p,\frac{1}{\l _n})$ and so, $\Delta (\wt{L})\subset
 \overline{B}_N(p,\frac{1}{\l _n})$. By definition of ${\cal A}_-$, this means that
 $\wt{L}\in {\cal A}_-$, from where Int$[\Delta (\wt{L})]\subset A_-$. This is impossible, as
 \[
 \de =d(Q,A_-)\leq d(Q,\wt{L})\leq d(Q,Q'),
 \]
which contradicts (P1). This contradiction finishes the proof of the assertion.
\end{proof}

\begin{assertion}
Item (2A) of Theorem~\ref{corol6.6} holds.
\end{assertion}
\begin{proof}
Consider a sequence $\l _n>0$ with $\l _n\to \infty $. As we have explained, there exists
$r_0>0$ small such that after extracting a subsequence, the weak CMC foliations $\l _n[
{\cal F}\cap B_N(p,r_0)]$ converge as $n\to \infty $ to a weak CMC foliation ${\cal F}_{\infty }$ of $\R^3-\{ \vec{0}\} $, and it only remains to show that ${\cal F}_{\infty }$ is not flat.
Consider the every $n\in \N$ the distance sphere $S^2_N(p,\frac{1}{\l _n})$.
By Assertion~\ref{ass7.15}, for $n$ sufficiently large there exists $L_n\in {\cal A}(r_0)$ such
that $L_n\subset \overline{B}_N(p,\frac{1}{\l _n})$ and $L_n\cap S^2_N(p,\frac{1}{\l _n})\neq \mbox{\O }$. Thus, the rescaled leaves $\l _nL_n$ of the weak CMC laminations
$\l _n[ {\cal F}\cap B_N(p,r_0)]$ stay inside the closed ball $\overline{B }_{\l _n N}(p,1)$ with a point in $S^2_{\l _nN}(p,1)$, and so, the limit foliation ${\cal F}_{\infty}$ has
a leaf contained in the closed unit ball of $\R ^3$ with some point in the unit sphere; in
particular, ${\cal F}_{\infty }$ is non-flat.
\end{proof}

\center{William H. Meeks, III at profmeeks@gmail.com\\
Mathematics Department, University of Massachusetts, Amherst, MA 01003}
\center{Joaqu\'\i n P\'{e}rez at jperez@ugr.es \qquad\qquad Antonio Ros at aros@ugr.es\\
Department of Geometry and Topology, University of Granada, Granada, Spain}

\end{document}